\documentclass[12pt]{amsart}

\usepackage{amssymb}
\usepackage{amsmath}
\usepackage{graphicx} 
\usepackage{here} 

\newcommand{\C}{\mathbb C}
\newcommand{\D}{\mathbb D}

\newcommand{\N}{\mathbb N}
\newcommand{\OO}{\mathbb O}
\newcommand{\R}{\mathbb R}

\newcommand{\Z}{\mathbb Z}

\renewcommand{\P}{\mathbb P}

\newcommand{\cF}{\mathcal F}

\newcommand{\cH}{\mathcal H}
\newcommand{\cI}{\mathcal I}

\newcommand{\cM}{\mathcal M}

\newcommand{\pa}{\partial}
\newcommand{\ex}{\mathbf{e}}

\renewcommand{\a}{\alpha} 
\renewcommand{\b}{\beta} 

\newcommand{\G}{\varGamma}
\newcommand{\De}{\mathit{\Delta}}
\newcommand{\de}{\delta}
\newcommand{\e}{\varepsilon}

\renewcommand{\l}{\lambda}

\newcommand{\z}{\zeta}
\newcommand{\w}{\omega}
\newcommand{\la}{\langle}
\newcommand{\ra}{\rangle}
\newcommand{\tr}{\;^t}

\newcommand{\diag}{\mathrm{diag}}

\newcommand{\one}{\mathbf{1}}

\newcommand{\rank}{\mathrm{rank}}
\newcommand{\reg}{\mathrm{reg}}
\newcommand{\init}{\mathrm{in}}
\newcommand{\sing}{\mathrm{Sing}}

\theoremstyle{plain}
\newtheorem{theorem}{Theorem}[section]
\newtheorem{proposition}{Proposition}[section]
\newtheorem{lemma}{Lemma}[section]
\newtheorem{cor}{Corollary}[section]

\theoremstyle{definition}

\newtheorem{remark}{Remark}[section]

\makeatletter
\@addtoreset{equation}{section}

\makeatother

\allowdisplaybreaks[1]

\newcommand{\HGF}[5]
{{}_{#1}F_{#2}\left(\begin{matrix}#3\\#4\end{matrix};#5\right)}

\newcommand{\hge}[2]
{\cF\left(\begin{matrix}#1\\#2\end{matrix}\right)}
\newcommand{\hgex}[3]
{\cF\left(\begin{matrix}#1\\#2\end{matrix};#3\right)}
\newcommand{\hgf}[3]
{F\left(\begin{matrix}#1\\#2\end{matrix};#3\right)}

\title
[A hypergeometric system of rank $9$]
{A system of hypergeometric differential equations in two variables 
 of rank $9$}

\author{Jyoichi Kaneko}
\address[Kaneko]{
   Department of Mathematical Sciences,
   University of the Ryukyus,
   Nishihara, Okinawa, 903-0213, Japan
}
\email{kaneko@math.u-ryukyu.ac.jp}

\author{Keiji Matsumoto}
\address[Matsumoto]{
Department of Mathematics\\
Hokkaido University\\
Sapporo 060-0810, Japan
}
\email{matsu@math.sci.hokudai.ac.jp}

\author{Katsuyoshi Ohara}
\address[Ohara]{
Faculty of Mathematics and Physics\\
Kanazawa University\\
Kanazawa 920-1192, Japan\\
}
\email{ohara@se.kanazawa-u.ac.jp
}


\date{\today}

\begin{document}
\maketitle
\begin{abstract}
We study a hypergeometric function in two variables and 
a system of hypergeometric differential equations 
associated with this function. This is a regular holonomic system of rank $9$. 
We give a fundamental system of solutions to this system 
in terms of this hypergeometric series.
We give circuit matrices along generators of 
the fundamental group of the complement of its singular locus
with respect to our fundamental system. 
\end{abstract}

\section{Introduction}
There are several generalizations of the original hypergeometric 
function $_2F_1$ and differential equation. For examples, 
we have the generalized hypergeometric function $_pF_{p-1}$, 
Appell's functions $F_1,\dots, F_4$, Lauricella's functions $F_A,\dots,F_D$, 
and the differential equations associated with them.  

In this paper, we study a hypergeometric function 
$\hgf{a}{B}{x}$ in two variables $x_1$ and $x_2$  
with parameters $a=(a_1,a_2,a_3)$, 
$B=\begin{pmatrix} b_1,b_2,1\\ b_3,b_4,1\end{pmatrix}$ defined in 
(\ref{eq:HGF}), 
which is one of generalizations of the hypergeometric series 
introduced by Kamp\'e de F\'eriet as mentioned in \cite{E}[\S1.5]. 
Our function can be  regarded as an extension of $_3F_2\left(
\begin{matrix} a_1,a_2,a_3\\ b_1,b_2\end{matrix};x\right)$ 
just like Appell's 
$F_4\left( \begin{matrix} a_1,a_2 \\ b_1, b_2  \end{matrix};x_1,x_2\right)$ 
as that of the original 
$_2F_1\left(\begin{matrix} a_1,a_2\\ b_1\end{matrix};x\right)$.

We give a system $\hge{a}{B}$ of differential equations satisfied by 
our function in Proposition \ref{prop:HGDE}. 
Our function and this system admit the $S_3\times D_4$-symmetry 
with respect to parameter and variable changes, 
where $S_3$ is the symmetric group of degree $3$ and 
$D_4$ is the dihedral group of order $8$. This symmetry helps us 
to find solutions and their integral representations, 
and to check some formulas.  
We show in Theorem \ref{th:rank} that 
the system $\hge{a}{B}$ is holonomic of rank $9$ and that  
its singular locus $S$ consists of the coordinate axes and 
a nodal cubic curve $R(x)=0$ triply tangent to them.
Under a $(\Z/(3\Z))^2$ covering map, 
the pull back of $S$ is decomposed into twelve lines 
in the complex projective plane $\P^2$.
These lines together with nine intersection points of them
form the Hesse configuration, in which lines and 
points satisfies that three points per line and four lines through each point.
In other words, we also have a holonomic system of rank $9$ 
defined on the complement of the Hesse configuration. 

We investigate in \cite{KMO1} the structure of 
the fundamental group $\pi_1(X,\dot x)$,  
where $X$ is the complement of the singular locus $S$.
We show that it is generated by three loops $\rho_1,\rho_2$ and 
$\rho_3$ turning around $x_1=0$, $x_2=0$ and $R(x)=0$, respectively, 
and that its is isomorphic to a group generated by three elements 
with four relations among them. 
To study the monodromy representation of $\hge{a}{B}$ in this paper, 
we use the generating loops $\rho_1,\rho_2$ and $\rho_3$ and 
three relations among them, 
which characterize an Artin group of infinite type.  

We construct a fundamental system of solutions to the system $\hge{a}{B}$ 
in Theorem \ref{th:fund-sols} by using the series $\hgf{a'}{B'}{x}$ with 
different parameters $a',B'$. 
We give their integral representations of 
Euler type in \S\ref{sec:Int-rep} by following results in \S4 in \cite{G1}. 
We consider the monodromy representation of our system in 
\S\ref{sec:Monodromy}. 
The problem is the computation of the circuit matrix 
along the loop $\rho_3$. 
We show that it has an eigenvector $v$ of eigenvalue $\lambda(\ne1)$ 
and the $8$-dimensional eigenspace of eigenvalue $1$ 
under some non-integral conditions on the parameters.
We express it as a reflection of root $v$  
with respect to an indeterminate form $H$. 
By normalizing $v$ to $(1,\dots,1)$ and restricting the 
system $\hge{a}{B}$ to $x_i=0$ ($i=1,2$), we determine the eigenvalue $\l$ 
and the form $H$ from the monodromy representation of 
the generalized hypergeometric differential equation $_3\cF_2$ 
as the method for Appell's system $\cF_C$ in \cite{M2}.
Note also that the system $\hge{a}{B}$ is regular singular 
by the behavior of its solutions around each component of 
the singular locus $S$.

Finally, we compute intersection numbers of twisted cycles given in 
\S\ref{sec:Int-rep}.
By these computations, we can conclude that 
the fundamental system of solutions used in the monodromy representation 
is given by the integral representations without Gamma factors.
As in \cite{G2}, \cite{GM}, \cite{M1} and \cite{MY} for Appell-Lauricella's 
systems, we express in Theorem \ref{th:monod-int} 
the circuit transforms in terms of intersection form 
independent of the choice of fundamental systems of solutions.  
It is studied in \cite{KMO1} that  
the monodromy representation of $\hge{a}{B}$ is irreducible 
under some non-integral conditions of parameters.

There is a further generalization 
$F_p^m\left(\begin{matrix} a\\B \end{matrix};x_1,\dots,x_m\right)$,  
which can be regarded as an $m$-variable version of $_pF_{p-1}$. 
This function satisfies 
a system $\cF_p^m\left(\begin{matrix} a\\B \end{matrix}\right)$ 
of differential equations of rank $p^m$. 
We study this function and system in the forthcoming paper \cite{KMO2}.

\section{A system of hypergeometric differential equations}
\label{sec:HGDE}
We define a hypergeometric series of two variables $x_1$ and $x_2$ 
with parameters $a=(a_1,a_2,a_3)$, 
$B= \begin{pmatrix} B_1\\ B_2 \end{pmatrix}=\begin{pmatrix}
b_{11},b_{12},b_{13}\\ b_{21},b_{22},b_{23}\end{pmatrix} 
=\begin{pmatrix}b_1,b_2,1\\ b_3,b_4,1\end{pmatrix}$ 
as 
\begin{align}
\label{eq:HGF}
& \hgf{a}{B}{x}=\hgf{a_1,a_2,a_3}{\begin{matrix} 
b_{11},b_{12},b_{13}\\b_{21},b_{22},b_{23}
\end{matrix}}
{x_1,x_2}
=\sum_{n\in \N^2} \prod_{j=1}^3 \dfrac{(a_j,n_1+n_2)}
{\prod_{i=1}^2(b_{ij},n_i)}x^n
\\
\nonumber
&=\sum_{n_1,n_2=0}^\infty
\frac{(a_1,n_1+n_2)(a_2,n_1+n_2)(a_3,n_1+n_2)}
{(b_1,n_1)(b_2,n_1)(1,n_1)(b_3,n_2)(b_4,n_2)(1,n_2)}x_1^{n_1}x_2^{n_2},
\end{align}
where $b_1,\dots,b_4\notin -\N=\{0,-1,-2,\dots\}$.
Note that this series reduces to 
$\HGF{3}{2}{a_1,a_2,a_3}{b_1,b_2}{x_1}$ and to 
$\HGF{3}{2}{a_1,a_2,a_3}{b_3,b_4}{x_2}$
when it is restricted to $x_2=0$ and to $x_1=0$, respectively. 
The following proposition is a direct consequence 
of the formula (29) in \S 1.3 in \cite{SK}.

\begin{proposition}
\label{prop:converge}
If $x$ belongs to the domain 
$$\D=\{(x_1,x_2)\in \C^2\mid\sqrt[3]{|x_1|}+\sqrt[3]{|x_2|}<1\},$$
then the hypergeometric series (\ref{eq:HGF}) 
absolutely converges. 
If $x$ does not belongs to the closure of $\D$, then it diverges. 
\end{proposition}

\begin{proposition}
\label{prop:F-symmetry}
The hypergeometric series (\ref{eq:HGF}) admits the symmetry:
$$
\hgf{a}{B}{x_1,x_2}
=
F\left(\begin{matrix}
a_{\sigma(1)},&a_{\sigma(2)},&a_{\sigma(3)}\\
b_{\tau(1),\sigma_{1}(1)},&b_{\tau(1),\sigma_{1}(2)},& 1\\
b_{\tau(2),\sigma_{2}(1)},&b_{\tau(2),\sigma_{2}(2)},& 1
\end{matrix}\ ;{x_{\tau(1)},x_{\tau(2)}}\right),
$$
where $\sigma$ belongs to the symmetric group $S_3$ of degree $3$, 
$\sigma_1,\sigma_2,\tau$ belong to the symmetric group $S_2$ of degree $2$. 
This symmetry is isomorphic to the direct product $S_3\times D_4$ 
of $S_3$ and the dihedral group $D_4$ of order $8$.
\end{proposition}
\begin{proof}
Since the numerator of the right hand side of (\ref{eq:HGF}) 
is symmetric with respect to the parameters $a_1,a_2,a_3$, 
$\hgf{a}{B}{x}$ is invariant under the action $\sigma \in S_3$.
We see the action on the parameters $b_1,\dots,b_4$ and the variables
$x_1,x_2$. Note that 
$$
\begin{matrix}
\sigma_{100}\cdot (b_1,b_2,b_3,b_4;x_1,x_2)
&=&(b_2,b_1,b_3,b_4;x_1,x_2),\\
\sigma_{010}\cdot (b_1,b_2,b_3,b_4;x_1,x_2)
&=&(b_1,b_2,b_4,b_3;x_1,x_2),\\
\sigma_{001}\cdot (b_1,b_2,b_3,b_4;x_1,x_2)
&=&(b_3,b_4,b_1,b_2;x_2,x_1),
\end{matrix}
$$
for three elements $\sigma_{100}$, $\sigma_{010}$, $\sigma_{001}$ 
of the triple of $S_2$ 
given by 
$(\sigma_1,\sigma_2;\tau)=((12),\mathrm{id};\mathrm{id})$, 
$(\mathrm{id},(12);\mathrm{id})$, $(\mathrm{id},\mathrm{id};(12))$.
By the symmetry of the right hand side of (\ref{eq:HGF}), 
$\hgf{a}{B}{x}$ is also invariant under the actions generated by them. 
We consider the group structure of the triple of $S_2$.
It is clear that $\sigma_{100}$ and $\sigma_{010}$ are commutative. 
By definition, we have 
\begin{align*}
&(\sigma_{100}\sigma_{001})\cdot (b_1,b_2,b_3,b_4;x_1,x_2)=
\sigma_{100}\cdot (b_3,b_4,b_1,b_2;x_1,x_2)\\
=&(b_4,b_3,b_1,b_2;x_2,x_1),\\
&(\sigma_{001}\sigma_{100})\cdot (b_1,b_2,b_3,b_4;x_1,x_2)=
\sigma_{001}\cdot (b_2,b_1,b_3,b_4;x_1,x_2)\\
=&(b_3,b_4,b_2,b_1;x_2,x_1),
\end{align*}
These imply that 
$$\sigma_{101}=((12),\mathrm{id};(12))=\sigma_{100}\sigma_{001}
\ne \sigma_{001}\sigma_{100}=(\mathrm{id},(12);(12))=\sigma_{011}.$$
Since $\sigma_{101}$ and $\sigma_{011}$ are of order $4$, this group 
is isomorphic to the dihedral group $D_4$  of order $8$.
By this action on $b_1,\dots,b_4$, 
we have an inclusion of $D_4$ into the symmetric group $S_4$ 
as in Table \ref{tab:D4S4}, 
where elements of $S_4$ are expressed in terms of cyclic permutations. 
It is easy to see that the action of $S_3$ commutes with that of $D_4$.
\end{proof}
\begin{table}[hbt]
$$
\begin{array}{|c|cccccccc|}
\hline
D_4&\sigma_{000} &\sigma_{100} &\sigma_{010} &\sigma_{110} &
   \sigma_{001} &\sigma_{101} &\sigma_{011} &\sigma_{111} \\
\hline
S_4&\mathrm{id} &(12) &(34) &(12)(34) &
    (13)(24) &(1423) &(1324) &(14)(23) \\
\hline
\end{array}
$$
\caption{Inclusion $D_4\hookrightarrow S_4$}
\label{tab:D4S4}
\end{table}

\begin{proposition}
\label{prop:HGDE}
The function $\hgf{a}{B}{x}$ satisfies hypergeometric differential equations
$$\Big[\prod_{j=1}^3(b_{ij}-1+\theta_i)\Big]\cdot f(x)=
\Big[x_i\prod_{j=1}^3(a_j+\theta_1+\theta_2)\Big]\cdot f(x)\quad (i=1,2),$$
where $f(x)$ is an unknown function, and 
$\theta_i=x_i\pa_i$, $\pa_i=\dfrac{\pa}{\pa x_i}$ $(i=1,2)$.
\end{proposition}
\begin{proof}
We write down the differential equations  as
\begin{align}
\label{eq:DE1}
& \theta_1(b_1-1+\theta_1)(b_2-1+\theta_1)f(x)\\
\nonumber
=&x_1(a_1+\theta_1+\theta_2)(a_2+\theta_1+\theta_2)
(a_3+\theta_1+\theta_2)f(x),\\[5mm]
\label{eq:DE2}
& \theta_2(b_3-1+\theta_2)(b_4-1+\theta_2)f(x)\\
\nonumber
=&x_2(a_1+\theta_1+\theta_2)(a_2+\theta_1+\theta_2)
(a_3+\theta_1+\theta_2)f(x). 
\end{align}
Let us show (\ref{eq:DE1}). Since 
\begin{align*}
\theta_1(b_1\!-\!1\!+\!\theta_1)(b_2\!-\!1\!+\!\theta_1)x_1^{n_1}x_2^{n_2}
&=n_1(b_1\!-\!1\!+\!n_1)(b_2\!-\!1\!+\!n_1)x_1^{n_1}x_2^{n_2},\\
x_1\prod_{i=1}^3(a_{i}\!+\!\theta_1\!+\!\theta_2)x_1^{n_1}x_2^{n_2}
&=\prod_{i=1}^3(a_{i}\!+\! n_1\!+\!n_2\!)x_1^{n_1+1}x_2^{n_2},
\end{align*}
we have 
\begin{align*}
&\theta_1(b_1-1+\theta_1)(b_2-1+\theta_1)\hgf{a}{B}{x}\\
=&\sum_{n_1=1,n_2=0}^\infty
\frac{(a_1,n_1+n_2)(a_2,n_1+n_2)(a_3,n_1+n_2)}
{(b_1,n_1\!-\!1)(b_2,n_1\!-\!1)(1,n_1\!-\!1)(b_3,n_2)(b_4,n_2)(1,n_2)}
x_1^{n_1}x_2^{n_2},\\
&x_1\prod_{i=1}^3(a_{i}\!+\!\theta_1\!+\!\theta_2)\hgf{a}{B}{x}\\
=&\sum_{n_1,n_2=0}^\infty
\frac{(a_1,n_1+n_2+1)(a_2,n_1+n_2+1)(a_3,n_1+n_2+1)}
{(b_1,n_1)(b_2,n_1)(1,n_1)(b_3,n_2)(b_4,n_2)(1,n_2)}x_1^{n_1+1}x_2^{n_2}.
\end{align*}
Note that these series coincide. Similarly, we can show (\ref{eq:DE2}).
\end{proof}

We study the system $\hge{a}{B}=\hgex{a}{B}{x}$ generated by 
the differential equations (\ref{eq:DE1}) and (\ref{eq:DE2}).

\begin{proposition}
\label{prop:E-symmetry}
The system $\hge{a}{B}$ admits the $S_3\times D_4$-symmetry:
$$
\hgex{a}{B}{x_1,x_2}
=
\cF\left(\begin{matrix}
a_{\sigma(1)},&a_{\sigma(2)},&a_{\sigma(3)}\\
b_{\tau(1),\sigma_{1}(1)},&b_{\tau(1),\sigma_{1}(2)},&1\\
b_{\tau(2),\sigma_{2}(1)},&b_{\tau(2),\sigma_{2}(2)},&1
\end{matrix}\ ;{x_{\tau(1)},x_{\tau(2)}}\right),
$$
for any $\sigma\in S_3$ and $\sigma_1,\sigma_2,\tau\in S_2$.
\end{proposition}
\begin{proof}
It is easy to see that the differential equations (\ref{eq:DE1}) and 
(\ref{eq:DE2}) are invariant under the action of $S_3$.
They are also invariant under the actions of $\sigma_{100}$ 
and $\sigma_{010}$, and exchanged by the action of $\sigma_{001}$, 
where  $\sigma_{100}$, $\sigma_{010}$ and $\sigma_{001}$ are 
given in the proof of Proposition \ref{prop:F-symmetry}.
\end{proof}

\begin{theorem}
\label{th:rank}
The system $\hge{a}{B}$ is of rank $9$. Its singular locus $S$ is 
$$
S=\{(x_1,x_2)\in \C^2\mid x_1x_2R(x_1,x_2)=0\},
$$
where 
\begin{align*}
R(x)&=R(x_1,x_2)=\prod_{k_1,k_2=0}^2 
(1-\w^{k_1}\sqrt[3]{x_1}-\w^{k_2}\sqrt[3]{x_2})\\
&=(1-x_1-x_2)^3-27x_1x_2,
\end{align*}
and $\w=\dfrac{-1+\sqrt{-3}}{2}.$
\end{theorem}

\begin{proof}
We can regard the system as a left ideal $I$ in 
the Weyl algebra $W_2=\C\la x_1,x_2,\pa_1,\pa_2\ra$.
Any element of $W_2$ 
has a canonical form 
$$\sum_{(i_1,i_2)} c_{i_1,i_2}(x_1,x_2)\pa_1^{i_1}\pa_2^{i_2},
\quad  c_{i_1,i_2}(x_1,x_2)\in \C[x_1,x_2].$$
By replacing $\pa_j$ with $\xi_j$ $(j=1,2)$ in canonical forms, 
we have a natural map from 
$W_2$ to the polynomial ring $\C[x_1,x_2,\xi_1,\xi_2]$.  
Then the initial ideal $\init_{w}\,(I)$ with respect to the weight
vector $w=(0,0,1,1)$ is defined as an ideal of $\C[x_1,x_2,\xi_1,\xi_2]$.
The zero set $V(\init_{w}\,(I))$ of the initial ideal in $\C^4$ is
called the characteristic variety of $I$.  
The singular locus $\sing(I)$ of the left ideal $I$ is defined by
the Zariski closure of the image of 
$V(\init_{w}\,(I))\setminus \{(x,\xi) \mid \xi_1=\xi_2=0\}$
under the projection $\C^4\ni (x,\xi)\mapsto x\in \C^2$.  
It is shown in \S1.4 of \cite{SST} 
that 
$\sing(I)$ coincides with the zero set $V(J)$ of the polynomial ideal
\[
J = ( \init_{w}\,(I) : 
\langle \xi_1, \xi_2\rangle^\infty ) \cap \C[x_1,x_2].
\]
Generators of the initial ideal $\init_{w}\,(I)$ are calculated from
the Gr\"obner basis of $I$ with respect to the weight vector $w$.
We can execute saturation and elimination of 
the ideal $I$ 
by using Risa/Asir (a computer algebra system). 
By this computation, it turns out that 
$J=\langle x_1^4x_2^4\,R(x)\rangle$.
Thus we have
\[
\sing\,(I) = V(J) = V(\sqrt{J}) = V(x_1x_2R(x)),
\]
where $\sqrt{J}$ is the radical of $J$.

The holonomic rank of $I$ is defined by 
$$\rank\,(I) = \dim_{\C(x_1,x_2)} (\mathcal{W}_2\,/\,\mathcal{W}_2I ),$$
where $\mathcal{W}_2 = \C(x_1,x_2)\langle \partial_1, \partial_2\rangle$,  
and the quotient $\mathcal{W}_2/\mathcal{W}_2I$ is regraded as a vector space over $\C(x_1,x_2)$.
With respect to the graded lexicographic order in $\mathcal{W}_2$,
a left $\mathcal{W}_2$-ideal $\mathcal{W}_2 I$ has a Gr\"obner basis 
whose initial terms are given by
$\{ \xi_1^5, \xi_1^3\xi_2, \xi_1\xi_2^2, \xi_2^3\}$.
Then the vector space $\mathcal{W}_2/\mathcal{W}_2 I$ has a basis 
$\{1,\pa_1, \pa_1^2, \pa_1^3, \pa_1^4, \pa_2, \pa_1 \pa_2, 
\pa_1^2 \pa_2,  \pa_2^2 \}$ consisting of $9$ monomials.  
Thus the holonomic rank of $I$ is $9$.
\end{proof}

\begin{remark}
\begin{enumerate}
\item
Since the affine curve $R(x)=0$ has a nodal singular point $(-1,-1)$, 
it is rational. In fact, it is expressed by a complex parameter as
$$(x,y)=(t^3,(1-t)^3),\quad t\in \C.$$
Via this parametrization, $t=-\w$ and $t=-\w^2$ 
correspond to the singular point $(-1,-1)$. 
\item 
Under the covering map 
$\C^2\ni (z_1,z_2) \mapsto (x_1,x_2)=(z_1^3,z_2^3)\in \C^2$, 
the pull back of the cubic curve $R(x)=0$ is decomposed into 
nine lines 
$$1-\w^{k_1}z_1-\w^{k_2}z_2=0,\quad (0\le k_1,k_2\le 2).$$
Thus the pull back of the singular locus $S$ consists of these lines 
together with $z_1=0$, $z_2=0$ and the line at infinity in 
the projective plane $\P^2$.  Note that there are nine intersection 
points $[\z_0,\z_1,\z_2]=[0,1,-\w^{k_1}],[1,0,\w^{k_1}],[1,\w^{k_1},0]$
$(0\le k_1\le 2)$ 
of two lines of them, where $[\z_0,\z_1,\z_2]$ are the projective coordinates 
with $(z_1,z_2)=(\z_1/\z_0,\z_2/\z_0)$. 
These twelve lines and the nine points form the Hesse configuration, 
in which lines and points satisfies that three points per line and 
four lines through each point.
\end{enumerate}
\end{remark}

We set 
$$
X=\{(x_1,x_2)\in \C^2\mid x_1x_2R(x_1,x_2)\ne 0\},
$$
which is the complement of the singular locus $S$ of the system $\hge{a}{B}$.
We select a base point $\dot x=(\e_1,\e_2)$ in $\D\cap X\cap \R^2$ 
with $0<\e_2<\e_1$, and take a small neighborhood $U$ of $\dot x$ in 
$\D\cap X$. 

\begin{theorem}
\label{th:fund-sols}
Suppose that 
$$
b_1,\ b_2,\ b_3,\ b_4,\ b_1-b_2,\ b_3-b_4,\ 
\notin \Z. 
$$
A fundamental system of solutions to $\hge{a}{B}$ on $U$ 
is given by 
\begin{equation}
\label{eq:fund-sols}
F_{jk}(x)=
  x_1^{1-b_{1j}}x_2^{1-b_{2k}}
\hgf{a+(2-b_{1j}-b_{2k})(e_1+e_2+e_3)}
{\begin{matrix}
B_1+(1-b_{1j})(e_j+e_1+e_2)\\
B_2+(1-b_{2k})(e_k+e_1+e_2)
  \end{matrix}}
{x}
\end{equation}
for $1\le j,k\le 3$, where $e_j$ is the $j$-th unit row vector of $\R^3$, and 
the indices $j$ and $k$ in $F_{jk}(x)$ are regarded 
as elements of $\{0,1,2\}=\Z/(3\Z)$.
They are arrayed in the order 
\begin{equation} 
\label{eq:order}
(jk)=(00),\ (10),\ (20),\ (01),\ (11),\ (21),\  (02),\ (12),\  (22) 
\end{equation}
as
 \begin{align*}
&\hgf{a_1,a_2,a_3}
{\begin{matrix}b_1,b_2,1\\ b_3,b_4,1
   \end{matrix} }{x},\\
x_1^{1-b_1}&F\left(\begin{matrix}
a_1-b_1+1,&a_2-b_1+1,&a_3-b_1+1\\
2-b_1,&b_2-b_1+1,&1\\
b_3,&b_4,&1
\end{matrix}
;x\right),\\
x_1^{1-b_2}&F\left(
\begin{matrix} a_1-b_2+1,&a_2-b_2+1,&a_3-b_2+1\\
b_1-b_2+1,&2-b_2,&1\\
b_3,&b_4,&1
\end{matrix}
;x\right),
\end{align*}
\begin{align*}
x_2^{1-b_3}&F
\left(\begin{matrix}
a_1-b_3+1,& a_2-b_3+1,&a_3-b_3+1\\
b_1,&b_2,&1\\
2-b_3,&b_4-b_3+1,&1
\end{matrix}
;x\right),\\
x_1^{1-b_1}x_2^{1-b_3}&F\left(\begin{matrix}
a_1\!-\!b_1\!-\!b_3\!+\!2,&a_2\!-\!b_1\!-\!b_3\!+\!2,
&a_3\!-\!b_1\!-\!b_3\!+\!2\\
2-b_1,&b_2-b_1+1,&1\\
2-b_3,&b_4-b_3+1,&1
\end{matrix}
;x\right),\\
x_1^{1-b_2}x_2^{1-b_3}&F\left(\begin{matrix}
a_1\!-\!b_2\!-\!b_3\!+\!2,&a_2\!-\!b_2\!-\!b_3\!+\!2,
&a_3\!-\!b_2\!-\!b_3\!+\!2\\
b_1-b_2+1,&2-b_2,&1\\
2-b_3,&b_4-b_3+1,&1
\end{matrix}
;x\right),
\end{align*}
\begin{align*}
x_2^{1-b_4} &F\left(
\begin{matrix}a_1-b_4+1,&a_2-b_4+1,&a_3-b_4+1\\
b_1,&b_2,&1\\
b_3-b_4+1,&2-b_4,&1
\end{matrix}
;x\right),\\
x_1^{1-b_1}x_2^{1-b_4}&F\left(\begin{matrix}
a_1\!-\!b_1\!-\!b_4\!+\!2,&a_2\!-\!b_1\!-\!b_4\!+\!2,
&a_3\!-\!b_1\!-\!b_4\!+\!2\\
2-b_1,&b_2-b_1+1,&1 \\
b_3-b_4+1,&2-b_4,&1
\end{matrix}
;x\right),\\
x_1^{1-b_2}x_2^{1-b_4}&F\left(\begin{matrix}
a_1\!-\!b_2\!-\!b_4\!+\!2,&a_2\!-\!b_2\!-\!b_4\!+\!2,
&a_3\!-\!b_2\!-\!b_4\!+\!2\\
b_1-b_2+1,&2-b_2,&1\\ 
b_3-b_4+1,&2-b_4,&1
\end{matrix}
;x\right).
\end{align*}
\end{theorem}

\begin{proof}
We can show that these functions are solutions to $\hge{a}{B}$
by using properties  
\begin{align*}
&\prod_{i=1}^3(a_{i}\!+\!\theta_1\!+\!\theta_2)x_1^{1-b_{1j}}x_2^{1-b_{2k}}
\!=\!x_1^{1-b_{1j}}x_2^{1-b_{2k}}\prod_{i=1}^3
(a_{i}\!-\!b_{1j}\!-\!b_{2k}\!+\!2\!+\!\theta_1\!+\!\theta_2),\\
&\prod_{i=1}^3(b_{1i}-1+\theta_1)x_1^{1-b_{1j}}x_2^{1-b_{2k}}
=x_1^{1-b_{1j}}x_2^{1-b_{2k}}\prod_{i=1}^3(b_{1i}-b_{1j}+\theta_1),\\
&\prod_{i=1}^3(b_{2i}-1+\theta_2)x_1^{1-b_{1j}}x_2^{1-b_{2k}}
=x_1^{1-b_{1j}}x_2^{1-b_{2k}}\prod_{i=1}^3(b_{2i}-b_{2k}+\theta_2),
\end{align*}
as elements of the (extended) Weyl algebra. Since the power functions
$x_1^{1-b_{1j}}x_2^{1-b_{2k}}$ are mutually different under our assumption, 
these functions are linearly independent.
\end{proof}
Recall that 
the series $\hgf{a}{B}{x}$ and the system $\hgex{a}{B}{x}$ are 
invariant under the action of $S_3\times D_4$.

\begin{proposition}
\label{prop:S-symmetry}
The group $S_3\times D_4$ 
acts on the fundamental system 
(\ref{eq:fund-sols}) of solutions to $\hge{a}{B}$ via the change of 
parameters and variables in Proposition \ref{prop:F-symmetry}.
There are three orbits $\{F_{00}(x)\}$, 
$$\{F_{10}(x),F_{20}(x),F_{01}(x),F_{02}(x)\},\quad
\{F_{11}(x),F_{21}(x),F_{12}(x),F_{22}(x)\}.$$
\end{proposition}
\begin{proof}
We have only to consider actions of $D_4$ on $B$ and $x$.
It is easy to see that every $F_{jk}(x)$ $(0\le j,k\le2)$ is changed into 
one of (\ref{eq:fund-sols}).  
By Proposition \ref{prop:F-symmetry}, $F_{00}(x)$ is invariant under 
this action. By the actions 
$\sigma_{100}$, $\sigma_{001}$ and $\sigma_{101}$
on $F_{10}(x)$, it changes into $F_{20}(x)$, $F_{01}(x)$ and $F_{02}(x)$, 
respectively; see Table \ref{tab:D4S4}.
For the orbit of $F_{11}(x)$, act 
$\sigma_{100}$, $\sigma_{010}$ and $\sigma_{110}$
on $F_{11}$.
\end{proof}

\begin{cor}
\label{cor:sol-at-infty}
We set 
$y=(y_1,y_2)=(\dfrac{-x_1}{x_2},\dfrac{1}{x_2})$, and 
suppose that 
$$
b_1,\ b_2,\ b_1-b_2,\ a_1-a_2,\ a_1-a_3,\ a_2-a_3\notin \Z.
$$ 
There are nine solutions to $\hge{a}{B}$ around a point 
$\dot y\in X$ near to $y=(0,0)$ as follows:
\begin{align*}
 y_2^{a_{i}}
&F\left(
\begin{matrix} a_{i},&a_{i}-b_3+1,&a_{i}-b_4+1\\
b_1,&b_2,&1\\
a_{i}-a_{j}+1,&a_{i}-a_{k}+1,&1
\end{matrix}
;y\right),\\
y_1^{1-b_1} y_2^{a_{i}}
&F\left(
\begin{matrix} a_{i}-b_1+1,&a_{i}-b_1-b_3+2,&a_{i}-b_1-b_4+2\\
2-b_1,&b_2-b_1+1,&1\\
a_{i}-a_{j}+1,&a_{i}-a_{k}+1,&1
\end{matrix}
;y\right),\\
y_1^{1-b_2} y_2^{a_{i}}
&F\left(
\begin{matrix} a_{i}-b_2+1,&a_{i}-b_2-b_3+2,&a_{i}-b_2-b_4+2\\
b_1-b_2+1,&2-b_2,&1\\
a_{i}-a_{j}+1,&a_{i}-a_{k}+1,&1
\end{matrix}
;y\right),
\end{align*}
where $i=1,2,3$ and $\{j,k\}=\{1,2,3\}-\{i\}$ as a set. 
\end{cor}
\begin{proof}
Since $y_1=-x_1/x_2$, $y_2=1/x_2$, we have
\begin{align*}
x_1\pa_1&=x_1\left(
\dfrac{\pa y_1}{\pa x_1}\dfrac{\pa} {\pa y_1}
+\dfrac{\pa y_2}{\pa x_1}\dfrac{\pa }{\pa y_2}
\right)=-\frac{x_1}{x_2}\dfrac{\pa} {\pa y_1},\\
 x_2\pa_2&=x_2\left(
\dfrac{\pa y_1}{\pa x_2}\dfrac{\pa} {\pa y_1}
+\dfrac{\pa y_2}{\pa x_2}\dfrac{\pa }{\pa y_2}
\right)=\dfrac{x_1}{x_2}\dfrac{\pa} {\pa y_1}
-\dfrac{1}{x_2}\dfrac{\pa} {\pa y_2}.
\end{align*}
Thus the operators $\theta_{y_1}=y_1\dfrac{\pa}{\pa y_1}$ and 
$\theta_{y_2}=y_2\dfrac{\pa}{\pa y_2}$ 
relate $\theta_1$ and $\theta_2$ as
$$\theta_1= \theta_{y_1},\quad 
\theta_2=-(\theta_{y_1}+\theta_{y_2}).
$$ 
Note that the system $\hge{a}{B}$ is generated by the equation (\ref{eq:DE2}) 
and 
$$x_2\theta_1(b_1\!-\!1+\!\theta_1)(b_2\!-\!1\!+\!\theta_1)f(x)=
x_1\theta_2(b_3\!-\!1+\!\theta_2)(b_4\!-\!1\!+\!\theta_2)f(x).$$
We express these differential equations 
in terms of $\theta_{y_1}$ and $\theta_{y_2}$: 
\begin{align*}
&(-a_1+\theta_{y_2})(-a_2+\theta_{y_2})
  (-a_3+\theta_{y_2})f(y)\\
=
& y_2
(\theta_{y_1}+\theta_{y_2})
(-b_3+1+\theta_{y_1}+\theta_{y_2})
(-b_4+1+\theta_{y_1}+\theta_{y_2})f(y),
\\[5mm]
&\theta_{y_1}(b_1-1+\theta_{y_1})
(b_2-1+\theta_{y_1})f(y)\\
=& 
y_1(\theta_{y_1}+\theta_{y_2})
(-b_3+1+\theta_{y_1}+\theta_{y_2})
(-b_4+1+\theta_{y_1}+\theta_{y_2})f(y).
\end{align*}
Substitute $f(y)=y_2^{\mu}g(y)$ into these differential equations. 
Then we have 
\begin{align*}
&(\mu-a_1+\theta_{y_2})(\mu-a_2+\theta_{y_2})
  (\mu-a_3+\theta_{y_2})g(y)\\
=
& y_2
(\mu+\theta_{y_1}+\theta_{y_2})
(\mu-b_3+1+\theta_{y_1}+\theta_{y_2})
(\mu-b_4+1+\theta_{y_1}+\theta_{y_2})g(y),
\\[5mm]
&\theta_{y_1}(b_1-1+\theta_{y_1})
(b_2-1+\theta_{y_1})g(y)\\
=& 
y_1(\mu+\theta_{y_1}+\theta_{y_2})
(\mu-b_3+1+\theta_{y_1}+\theta_{y_2})
(\mu-b_4+1+\theta_{y_1}+\theta_{y_2})g(y).
\end{align*}
If $\mu=a_i$ and 
$g(y)=
F\left(
\begin{matrix} a_{i},&a_{i}-b_3+1,&a_{i}-b_4+1\\
b_1,&b_2,&1\\
a_{i}-a_{j}+1,&a_{i}-a_{k}+1,&1
\end{matrix}
;y\right)$ then $y_2^\mu g(y)$ satisfies these differential 
equations. To obtain the rests, consider the series with 
exponents $y_1^{1-b_1}$ and $y_1^{1-b_2}$ in Theorem 
\ref{th:fund-sols}. 
\end{proof}

\section{Integral representations of Euler type}
\label{sec:Int-rep}
\begin{theorem}\label{th:int-rep}
The hypergeometric series $\hgf{a}{B}{x}$ admits the integral representation 
of Euler type:
\begin{equation}
\label{eq:int-rep}
C_{00}\cdot \int_{\De_{00}} u(t,x)dt,
\end{equation}
where $u(t,x)$ is 
$$
\Big(\prod_{j=1}^4 t_j^{-b_j}\Big)
(1-t_1-t_3)^{b_1+b_3-a_1-2}(1-t_2-t_4)^{b_2+b_4-a_2-2}
\Big(1-\frac{x_1}{t_1t_2}-\frac{x_2}{t_3t_4}\Big)^{-a_3},
$$
$dt= dt_1\wedge dt_3\wedge dt_2\wedge dt_4$, 
and the gamma factor $C_{00}$ is 
$$\frac{\G(1-a_1)\G(1-a_2)}{\G(1\!-\! b_1)\G(1\!-\! b_3)
\G(b_1\!+\! b_3\!-\! a_1\!-\! 1)
\G(1\!-\! b_2)\G(1\!-\! b_4)\G(b_2\!+\!b_4\!-\! a_2\!-\!1)}.
$$
Here $\De_{00}$ is a $4$-chain $\reg_{00}^c(\triangle_1\times \triangle_2)$ 
of the regularization of the direct product of triangles
\begin{align*}
\triangle_1&=\{(t_1,t_3)\in \R^2\mid t_1>0,\ t_3>0,\ 1-t_1-t_3>0\},\\
\triangle_2&=\{(t_2,t_4)\in \R^2\mid t_2>0,\ t_4>0,\ 1-t_2-t_4>0\},
\end{align*} 
associated with 
$$u(t,0,0)=t_1^{-b_1}t_3^{-b_3}(1\!-\! t_1\!-\! t_3)^{b_1+b_3-a_1-2}
 t_2^{-b_2}t_4^{-b_4}(1\!-\! t_2\!-\!t_4)^{b_2+b_4-a_2-2}, 
$$
and $x=(x_1,x_2)$ is supposed to be so close to $(0,0)$ that 
the hypersurface 
$$Q=\{t\in \C^4\mid t_1t_2t_3t_4-t_3t_4x_1 -t_1t_2x_2=0\}$$
does not intersect with any component of $\De_{00}$.
\end{theorem}

\begin{remark}
\label{rem:regularization}
\begin{enumerate}
\item 
We set 
$$\ell_{ij}(t)=1-t_i-t_j\ (1\le i<j\le  4),\quad  
q(t)=1-\frac{x_1}{t_1t_2}-\frac{x_2}{t_3t_4},
$$
and 
\begin{equation}
\label{eq:int-space}
T=\{t\in \C^4\mid t_1t_2t_3t_4\ell_{13}(t)\ell_{24}(t)q(t)\ne 0\},
\end{equation}
which depends on $x\in \C^2$. 
For any fixed $x\in X$, a locally constant sheaf $\mathcal{L}_u$ 
on $T$ is given by $u(t,x)$. 
Twisted homology groups $H_k(T,\mathcal{L}_u)$ 
are defined from a complex of chains with sections of $\mathcal{L}_u$.  
Similarly, locally finite ones $H_k^{lf}(T,\mathcal{L}_u)$ are defined.
It is known that the natural map from $H_4(T,\mathcal{L}_u)$ to 
$H_4^{lf}(T,\mathcal{L}_u)$ becomes isomorphic under some non-integral 
conditions on the parameters. In this case, 
the regularization $\reg$ is defined by its inverse.
It is different from the regularization $\reg_{00}$ 
in the construction of $\De_{00}$, since the function $u(t,0,0)$ is used.
However, by loading a branch of $u(t,x)$ 
on every component of $\De_{00}$, we have an element $\De_{00}^u$ of 
$H_4(T,\mathcal{L}_u)$ for any $x$ near to $(0,0)$. 
We can also make the continuation of any element 
$\De^u \in H_4(T,\mathcal{L}_u)$ along a path in $X$ 
starting from $\dot x=(\e_1,\e_2)$.
\item 
We can construct $\reg_{00}^c(\triangle_1\times \triangle_2)$ as
the direct product of $2$-chains given by the 
regularizations of the triangles 
\begin{align*}
\triangle_1&=\{(t_1,t_3)\in \R^2\mid t_1>0,\ t_3>0,\ 1-t_1-t_3>0\},\\
\triangle_2&=\{(t_2,t_4)\in \R^2\mid t_2>0,\ t_4>0,\ 1-t_2-t_4>0\},
\end{align*} 
associated with  
\begin{align*}
u_1&=t_1^{-b_1}t_3^{-b_3}(1\!-\! t_1\!-\! t_3)^{b_1+b_3-a_1-2},\\
u_2&=t_2^{-b_2}t_4^{-b_4}(1\!-\! t_2\!-\!t_4)^{b_2+b_4-a_2-2}, 
\end{align*}
see Figure \ref{fig:cycle}. 
Refer to \S 3.2.4 of \cite{AK} and \S2 in \cite{GM} 
for an explicit construction of the regularization. 
If $x$ is close to $(0,0)$, then the hypersurface 
$Q$ passes through the tuber neighborhood of the hyperplanes 
$t_1=0,\dots,t_4=0$, which is made in the construction of 
the regularization of $\triangle_1\times \triangle_2$. 
Thus we can assume that 
$$\big|\frac{x_1}{t_1t_2}+\frac{x_2}{t_3t_4}\big|<1$$
for any $t$ in each component of $\De_{00}$.
\end{enumerate}
\end{remark}
\begin{figure}
\includegraphics[width=10cm]{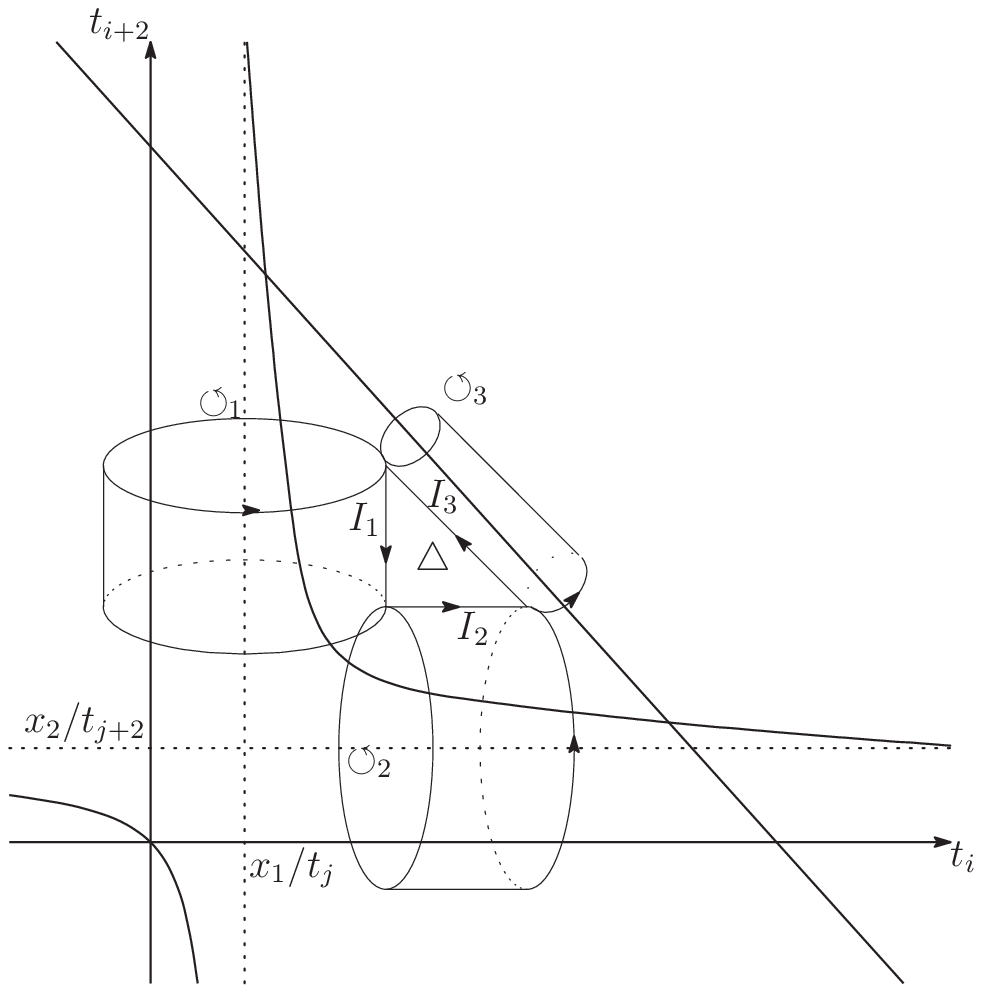}
\caption{Regularizations of $\triangle_i$ ($i=1,2$, $\{i,j\}=\{1,2\}$)}
\label{fig:cycle}
\end{figure}

\begin{proof}
Since $x$ is close to $(0,0)$, the factor including $x_1$ and $x_2$ 
admits the expansion 
\begin{align*}
& \Big(1-\frac{x_1}{t_1t_2}-\frac{x_2}{t_3t_4}\Big)^{-a_3}
=\sum_{N=0}^{\infty} \frac{(a_3,N)}{N!}
\Big(\frac{x_1}{t_1t_2}+\frac{x_2}{t_3t_4}\Big)^N\\
=&\sum_{N=0}^{\infty} \frac{(a_3,N)}{N!}
\sum_{n_1+n_2=N}\binom{N}{n_1}
\Big(\frac{x_1}{t_1t_2}\Big)^{n_1}\Big(\frac{x_2}{t_3t_4}\Big)^{n_2}\\
=&\sum_{n_1,n_2=0}^\infty \frac{(a_3,n_1+n_2)}{n_1!n_2!}
(t_1t_2)^{-n_1}(t_3t_4)^{-n_2}\cdot x_1^{n_1}x_2^{n_2}.
\end{align*}
Thanks to Remark \ref{rem:regularization} (2), this expansion is valid 
on each component of $\De_{00}$.
Change the order of the summation and the integration.
The coefficient of $x_1^{n_1}x_2^{n_2}$ in the integration 
(without the gamma factor $C_{00}$) is the product of 
\begin{align*}
& \int_{\triangle_1}t_1^{-b_1-n_1}t_3^{-b_3-n_2}(1-t_1-t_3)^{b_1+b_3-a_1-2}
dt_1\wedge dt_3\\
&=
\frac{\G(1-b_1-n_1)\G(1-b_3-n_2)\G(b_1+b_3-a_1-1)}
{\G(1-a_1-n_1-n_2)}
\end{align*}
and 
\begin{align*}
& \int_{\triangle_2}t_2^{-b_2-n_1}t_4^{-b_4-n_2}(1-t_2-t_4)^{b_2+b_4-a_2-2}
dt_2\wedge dt_4\\
&=
\frac{\G(1-b_2-n_1)\G(1-b_4-n_2)\G(b_2+b_4-a_2-1)}
{\G(1-a_2-n_1-n_2)}.
\end{align*}
By using the formulas 
$(\a,m)\G(\a)=\G(\a+m)$, and  
$\G(\a)\G(1-\a)=\dfrac{\pi}{\sin(\pi\a)}$, 
we rewrite $\G(1-b_1-n_1)$ to
$$\G(1-b_1-n_1)
=\frac{\pi}{(-1)^{n_1}\sin(\pi b_1)\G(b_1)(b_1,n_1)}
=\frac{\G(1-b_1)}{(-1)^{n_1}(b_1,n_1)}.
$$
Similarly we have 
\begin{align*}
\G(1-b_3-n_2)&=\frac{\G(1-b_3)}{(-1)^{n_2}(b_3,n_2)},\\
\G(1-a_1-n_1-n_2)&=
\frac{\G(1-a_1)}{(-1)^{n_1+n_2}(a_1,n_1+n_2)},
\end{align*}
which yield 
\begin{align*}
& \frac{\G(1-b_1-n_1)\G(1-b_3-n_2)\G(b_1+b_3-a_1-1)}{\G(1-a_1-n_1-n_2)}\\
&=\frac{\G(1-b_1)\G(1-b_3)\G(b_1+b_3-a_1-1)}{\G(1-a_1)}
\cdot \frac{(a_1,n_1+n_2)}{(b_1,n_1)(b_3,n_2)}.
\end{align*}
Since the other gamma factor is 
$$
\frac{\G(1-b_2)\G(1-b_4)\G(b_2+b_4-a_2-1)}{\G(1-a_2)}
\cdot \frac{(a_2,n_1+n_2)}{(b_2,n_1)(b_4,n_2)},
$$
the integral representation is obtained. 
\end{proof}

\begin{remark}
By acting the group $S_3\times D_4$ on (\ref{eq:int-rep}), 
we have $48$ integral representations of $\hgf{a}{B}{x}$.
\end{remark}

\begin{cor}
\label{cor:Euler}
The Euler number $\chi(T)$ is nine for any $x\in X$.
\end{cor}
\begin{proof}
Thanks to Theorem 1 in \cite{C},  only the $4$-th twisted homology 
$H_4(T,\mathcal{L}_u)$ survives. Thus $\chi(T)$ is equal to 
the rank of $H_4(T,\mathcal{L}_u)$. 
There is a natural isomorphism from this homology group to 
the vector space of local solutions to $\hge{a}{B}$ around $x$. 
Hence we have $\chi(T)=9$ for any $x\in X$ by Theorem \ref{th:rank}.
\end{proof}

\begin{cor}
\label{cor:2nd-int}
The solution $F_{10}(x)$ in Theorem \ref{th:fund-sols} is expressed as
$$C_{10}\int_{\De_{10}} u(t,x)dt,$$
where the gamma factor $C_{10}$ is 
\begin{align*}
& \frac{\G(b_1-a_1)\G(b_1-a_2)\G(b_1-a_3)}
{\G(b_1+b_3-a_1-1)\G(b_2+b_4-a_2-1)\G(1-a_3)}\\
\cdot& \frac{1}{\G(b_1-1)\G(b_1-b_2)\G(1-b_3)\G(1-b_4)},
\end{align*}
and $\De_{10}$ is the image of $-\reg_{10}^c(\square_{1}\times \triangle_2)$ 
under the map 
$$\imath_{10}:(s_1,s_2,s_3,s_4) \mapsto (t_1,t_2,t_3,t_4)=(\frac{x_1}{s_1s_2}
,s_2,s_3,s_4).$$
Here $\reg_{10}^c(\square_{1}\times \triangle_2)$ is 
a $4$-chain of the regularization of the direct product of 
$$\square_{1}=\{(s_1,s_3)\in \R^2\mid 0<s_1<1, \ 0<s_3<1\}$$
and $\triangle_2$ associated with   
\begin{align*}
&
s_1^{b_1-2}(1\!-\!s_1)^{-a_3}s_3^{-b_3}(1\!-\!s_3)^{b_1+b_3-a_1-2}
\cdot s_2^{b_1-b_2-1}s_4^{-b_4}(1\!-\! s_2\!-\!s_4)^{b_2+b_4-a_2-2},
\end{align*}
and the point  $x=(x_1,x_2)$ is supposed to be so close to $(0,0)$ that 
the pull back hypersurface $\imath_{10}^*Q$ does not intersect with 
any component of $\reg_{10}^c(\square_1\times \triangle_2)$.
\end{cor}
\begin{proof}
Apply the variable change $t=\imath_{10}(s)$  
to $\int_{\De_{10}}u(t,x)dt$ with paying attention to the signs in  
$\De_{01}=\imath_{10}(-\reg_{10}^c(\square_{1}\times \triangle_2))$ 
and $\imath_{10}^*(dt)$.
Then it changes into 
\begin{align*}
 & x_1^{1-b_1}\int_{\reg_{10}^c(\square_1\times \triangle_2)}
 s_1^{b_1-2}s_2^{b_1-b_2-1}s_3^{-b_3}s_4^{-b_4}
 (1-s_1-\frac{x_2}{s_3s_4})^{-a_3}\\
 & \hspace{2cm}\cdot (1-\frac{x_1}{s_1s_2}-s_3)^{b_1+b_3-a_1-2}
 (1-s_2-s_4)^{b_2+b_4-a_2-2}ds,
\end{align*}
where $ds=ds_1\wedge ds_3\wedge ds_2\wedge ds_4$. 
Expand the factors including $x_1$ and $x_2$. Then we have
\begin{align*}
\hspace{-2mm}& (1-s_3-\dfrac{x_1}{s_1s_2})^{b_1+b_3-a_1-2}
=\{(1\!-\! s_3)\cdot(1\!-\!\dfrac{x_1}{(1\!-\! s_3)s_1s_2})\}^{b_1+b_3-a_1-2}\\
\hspace{-2mm}&=
(1-s_3)^{b_1+b_3-a_1-2}\sum_{n_1=0}^\infty\dfrac{(a_1-b_1-b_3+2,n_1)}{n_1!}
\dfrac{x_1^{n_1}}{(1-s_3)^{n_1}s_1^{n_1}s_2^{n_1}},
\end{align*}
\begin{align*}
& (1-s_1-\dfrac{x_2}{s_3s_4})^{-a_3}=(1-s_1)^{-a_3}\cdot
(1-\dfrac{x_2}{(1-s_1)s_3s_4})^{-a_3}\\
&=(1-s_1)^{-a_3}\sum_{n_2=0}^\infty \dfrac{(a_3,n_2)}{n_2!}
\frac{x_2^{n_2}}{(1-s_1)^{n_2}s_3^{n_2}s_4^{n_2}}.
\end{align*}
Thus the integrand is 
\begin{align*}
&  \sum_{(n_1,n_2)\in \N^2}x_1^{n_1}x_2^{n_2}
\dfrac{(a_1-b_1-b_3+2,n_1)}{n_1!}
\dfrac{(a_3,n_2)}{n_2!}
\\
& \hspace{3cm}\cdot 
s_1^{b_1-2-n_1}(1-s_1)^{-a_3-n_2}s_3^{-b_3-n_2}(1-s_3)^{b_1+b_3-a_1-2-n_1}
\\
& \hspace{3cm}\cdot
s_2^{b_1-b_2-1-n_1}s_4^{-b_4-n_2}(1-s_2-s_4)^{b_2+b_4-a_2-2}.
\end{align*}
Integrate this over $\reg_{10}^c(\square_1\times \triangle_2)$, 
then we have 
\begin{align*}
&  \sum_{(n_1,n_2)\in \N^2}x_1^{n_1}x_2^{n_2}
\dfrac{(a_1-b_1-b_3+2,n_1)}{n_1!}
\dfrac{(a_3,n_2)}{n_2!}
\\
& \hspace{2.5cm}\cdot 
\dfrac{\G(b_1-1-n_1)\G(-a_3-n_2+1)}{\G(b_1-a_3-n_1-n_2)}\\
& \hspace{2.5cm}\cdot 
\dfrac{\G(-b_3-n_2+1)\G(b_1+b_3-a_1-n_1-1)}{\G(b_1-a_1-n_1-n_2)}
\\
& \hspace{2.5cm}\cdot
\dfrac{\G(b_1-b_2-n_1)\G(-b_4-n_2+1)\G(b_2+b_4-a_2-1)}{\G(b_1-a_2-n_1-n_2)}.
\end{align*}
By rewriting each Gamma factor including $-n_1$, $-n_2$, $-n_1-n_2$
into the product of the Gamma factor and Pochhammer's symbol,  
we have the second solution in Theorem \ref{th:fund-sols}.
Here note that Pochhammer's symbols $(a_1-b_1-b_3+2,n_1)$ and $(a_3,n_2)$ 
are canceled with $\G(b_1+b_3-a_1-1-n_1)$ and $\G(-a_3-n_2+1)$. 
\end{proof}

\begin{figure}
\begin{center}
\includegraphics[width=10cm]{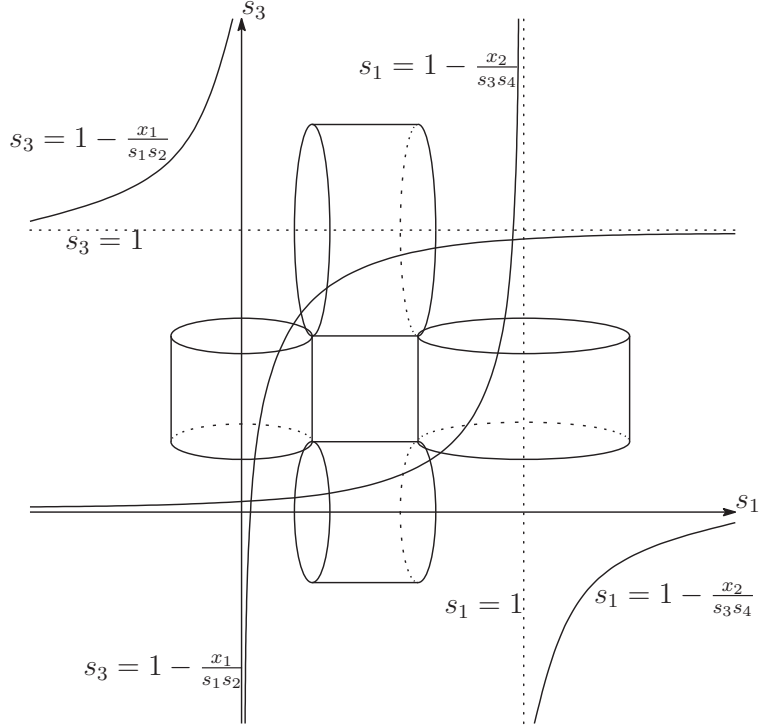}
\end{center}
\caption{Regularization of $\square_{1}$}
\label{fig:D13}
\end{figure}

\begin{remark}
\label{rem:Delta10}
\begin{enumerate}
\item 
A twisted cycle $\De_{10}^u$ is defined by 
the pair of $\De_{10}$ and a branch of $u(t,x)$ on it. 

\item
We can construct $\reg_{10}^c(\square_1\times \triangle_2)$ as
the direct product of $2$-chains given by the 
regularizations of $\square_1$ and $\triangle_2$
associated with  
\begin{align*}
u'_1&=s_1^{b_1-2}(1\!-\!s_1)^{-a_3}s_3^{-b_3}(1\!-\! s_3)^{b_1+b_3-a_1-2},
\\
u'_2&= s_2^{b_1-b_2-1}s_4^{-b_4}(1\!-\! s_2\!-\!s_4)^{b_2+b_4-a_2-2},
\end{align*}
see Figure \ref{fig:D13} for the regularization of $\square_1$. 

\item
By acting $((12),\sigma_{110})\in S_3\times D_4$ 
on Corollary \ref{cor:2nd-int}, and 
using the variable change 
$(t_1,t_2,t_3,t_4)\mapsto (t_2,t_1,t_4,t_3),$
we have an Euler type integral 
$$C_{20}\int_{\De_{20}} u(t,x)dt$$ 
of the solution 
$F_{20}(x)$ in Theorem \ref{th:fund-sols}. 
A twisted cycle $\De_{20}^u$ is defined by 
the pair of $\De_{20}$ and a branch of $u(t,x)$ on it. 
\item 
We have Euler type integrals 
$$C_{0k}\int_{\De_{0k}} u(t,x)dt\quad (k=1,2)
$$ 
of the solutions 
$F_{0k}(x)$ in Theorem \ref{th:fund-sols}, and twisted cycles 
$\De_{0k}^u$ by the action 
$(\mathrm{id},\sigma_{001})\in S_3\times D_4$ and the variable change 
$t\mapsto (t_3,t_4,t_1,t_2)$ to those of $F_{k0}(x)$.
\end{enumerate}
\end{remark}

Let $\imath_{11}$ be an involution given by 
$$
\imath_{11}:T\ni (s_1,s_2,s_3,s_4)\mapsto 
t=\Big(\dfrac{x_1}{s_1s_2},s_2,\dfrac{x_2}{s_3s_4},s_4\Big)\in T.
$$
Note that 
$$\imath_{11}^*(\ell_{13}(t))=q(s),\quad 
\imath_{11}^*(\ell_{24}(t))=\ell_{24}(s),\quad 
\imath_{11}^*(q(t))=\ell_{13}(s).$$
We define a $4$-chain $\De_{11}$ as the image of 
$\reg_{11}^c(\triangle_1\times \triangle_2)$ 
associated with
$$
s_1^{b_1-2}s_3^{b_3-2}(1\!-\!s_1\!-\!s_3)^{-a_3}\cdot 
s_2^{b_1-b_2-1}s_4^{b_3-b_4-1}(1\!-\!s_2\!-\! s_4)^{b_2+b_4-a_2-2}
$$
under the involution $\imath_{11}$.

\begin{cor}
\label{cor:5th-int} 
The solution $F_{11}(x)$ in Theorem \ref{th:fund-sols} is expressed as
$$
C_{11}\cdot \int_{\De_{11}} u(t,x)dt,
$$
where $C_{11}$ is 
$$
C_{11}=
\frac{\G(b_1\!+\!b_3\!-\!a_3\!-\!1) \G(b_1\!+\!b_3\!-\!a_2\!-\!1)}
{\G(b_1\!-\! 1)\G(b_3\!-\! 1)\G(1\!-\! a_3)
\G(b_1\!-\! b_2)\G(b_3\!-\! b_4)\G(b_2\!+\!b_4\!-\! a_2\!-\!1)}.
$$
\end{cor}
\begin{proof}
By the variable change $\imath_{11}$, 
the integral changes to 
\begin{align*}
& C_{11}\cdot x_1^{1-b_1}x_2^{1-b_3}
\int_{\reg_{11}^c(\triangle_1\times \triangle_2)} 
s_1^{b_1-2}
s_2^{b_1-b_2-1}
s_3^{b_3-2}
s_4^{b_3-b_4-1}\\ &\hspace{40mm}
\cdot 
\ell_{13}(s)^{-a_3}q(s)^{b_1+b_3-a_1-2}\ell_{24}(s)^{b_2+b_4-a_2-2} ds.
\end{align*}
We have only to consider the change of parameters. 
Here note that $\hgf{a}{B}{x}$ is symmetric with respect to $a_1,a_2,a_3$. 
\end{proof}
\begin{remark}
\label{rem:Delta11}
\begin{enumerate}
\item A twisted cycle $\De_{11}^u$ is defined by the pair of $\De_{11}$ and 
a branch of $u(t,x)$ on it.

\item By using the action of $\sigma_{110}\in D_4$ 
on Corollary \ref{cor:5th-int}, we have 
an Euler type integral of the solution $F_{22}(x)$ 
in Theorem \ref{th:fund-sols}. 
Its integrand can be transformed 
into $u(t,x)$ by the variable change $t \mapsto(t_3,t_4,t_1,t_2)$, and 
a twisted cycle $\De_{22}^u$ is obtained.
\end{enumerate}
\end{remark}

Let $\imath_{21}$ be the variable change 
$$
\imath_{21}:(s_1,s_2,s_3,s_4)\mapsto 
t=\Big(s_1,\dfrac{x_1}{s_1s_2},\dfrac{x_2}{s_3s_4},s_3\Big)\in T.
$$
We define a $4$-chain $\De_{21}$ as the image of 
$-\reg_{21}^c(\square_1\times \triangle_2)$ associated with 
$$s_1^{b_2\!-\!b_1\!-\!1}(1\!-\!s_1)^{b_1\!+\!b_3\!-\! a_1\!-\!2}
s_3^{b_3\!-\!b_4\!-\!1}(1\!-\!s_3)^{b_2\!+\!b_4\!-\!a_2\!-\!2}\cdot 
s_2^{b_2\!-\!2}s_4^{b_3\!-\!2}(1\!-\! s_2\!-\! s_4)^{-a_3}
$$
under the map $\imath_{21}$.

\begin{cor}
\label{cor:6th-int} 
The solution $F_{21}(x)$ in Theorem \ref{th:fund-sols} is expressed as
$$
C_{21}\cdot \int_{\De_{21}} u(t,x)dt,
$$
where the gamma factor $C_{21}$ is
\begin{align*}
&\frac{\G(b_2\!+\!b_3\!-\!a_1\!-\!1) \G(b_2\!+\!b_3\!-\!a_2\!-\!1) 
\G(b_2\!+\!b_3\!-\!a_3\!-\!1)}
{\G(b_1\!+\!b_3\!-\! a_1\!-\!1)\G(b_2\!+\!b_4\!-\! a_2\!-\!1)\G(1\!-\!a_3)}\\
&\cdot \frac{1}{\G(b_2\!-\! 1)\G(b_3\!-\! 1)\G(b_2\!-\! b_1)\G(b_3\!-\! b_4)}.
\end{align*}
\end{cor}
\begin{proof}
Apply the variable change $\imath_{21}$ to the integral with paying attention
to the signs in 
$\De_{21}=\imath_{21}(-\reg_{21}^c(\square_1\times \triangle_2))$ and 
$\imath_{21}^*(dt)$. 
Then it becomes
\begin{align*}
&x_1^{1-b_2}x_2^{1-b_3}
\int_{\reg_{21}^c(\square_1\times \triangle_2)} 
s_1^{b_2\!-\! b_1\!-\!1}s_2^{b_2\!-\!2}s_3^{b_3\!-\!b_4\!-\!1}s_4^{b_3\!-\!2}
(1\!-\!s_1\!-\!\frac{x_2}{s_3s_4})^{b_1\!+\! b_3\!-\! a_1\!-\!2}\\
&\hspace{3.3cm}\cdot (1-\frac{x_1}{s_1s_2}-s_3)^{b_2+b_4-a_2-2}
(1-s_2-s_4)^{-a_3}ds.
\end{align*}
To express this integral by the hypergeometric function, 
follow the proof of Corollary \ref{cor:2nd-int}. 
\end{proof}

\begin{remark}
\label{rem:Delta21}
\begin{enumerate}
\item A twisted cycle $\De_{21}^u$ is defined by the pair of $\De_{21}$ and 
a branch of $u(t,x)$ on it.

\item By using the action of $\sigma_{110}\in D_4$ 
on Corollary \ref{cor:5th-int}, we have 
an Euler type integral of the solution $F_{12}(x)$ 
in Theorem \ref{th:fund-sols}. 
Its integrand can be transformed 
into $u(t,x)$ by the variable change $t \mapsto(t_3,t_4,t_1,t_2)$, and 
a twisted cycle $\De_{12}^u$ is obtained.
\end{enumerate}
\end{remark}

\section{Fundamental group of the complement of the singular locus}
\label{sec:pi-1}
Recall that the base point $\dot x=(\e_1,\e_2)\in X$ is chosen so that 
$\e_1$ and $\e_2$ are small positive real numbers satisfying $0<\e_2<\e_1$. 
Let $\rho_1$ and $\rho_2$ be loops given by 
\begin{align*}
& \rho_1:[0,1]\ni r \mapsto (\e_1\exp(2\pi\sqrt{-1}r),\e_2)\in X,\\
& \rho_2:[0,1]\ni r \mapsto (\e_1,\e_2\exp(2\pi\sqrt{-1}r))\in X.
\end{align*}
Let $L$ be the complex line passing through $(0,0)$ and $\dot x$, 
which is expressed as
$$L=\{(\e_1r,\e_2r)\in \C^2\mid r\in \C\}.$$
We regard $r\in \C$ as a coordinate of $L$.
The intersection of $L$ and the curve $S_3:R(x)=0$ consists of three points
$P_1,P_2,P_3$. Let $P_1$ be the point whose coordinate $r_1\in \C$ is real.
Let $\rho_3$ be a loop in $L$ starting from $\dot x$, 
approaching $P_1$ along the real axis, 
turning once around $P_1$ positively, and tracing back to $\dot x$. 

We use the same symbol $\rho$ for the element of the 
fundamental group $\pi_1(X,\dot x)$ represented by a loop $\rho$ with 
terminal $\dot x$.
The structure of $\pi_1(X,\dot x)$ is studied in \cite{KMO1}.
It is shown that these loops generate $\pi_1(X,\dot x)$, and that 
it is isomorphic to a group generated by three elements with 
four relations among them. 
In this paper, we use a relation 
$$\rho_1\cdot \rho_2=\rho_2\cdot \rho_1$$
and that in the following proposition. 
\begin{proposition}
\label{prop:reduction}
The element
$$\rho_3\cdot (\rho_2\cdot \rho_3\cdot \rho_2^{-1})\cdot
(\rho_2^2\cdot \rho_3\cdot\rho_2^{-2})$$
commutes with $\rho_2$ as elements of $\pi_1(X,\dot x)$, 
where $\rho_2\cdot\rho_3$ is the loop joining $\rho_2$ to $\rho_3$.
Moreover, this loop is homotopic to the loop 
$$\rho_2':[0,1]\ni r \mapsto (1+(\e_1-1)\exp(2\pi\sqrt{-1}r),0)\in 
\{(x_1,x_2)\mid x_2=0\} \subset \C^2$$
in the space 
$$\{(x_1,x_2)\in \C^2\mid x_1R(x_1,x_2)\ne 0\}.$$
\end{proposition}
\begin{proof}
Let the line $L$ move along the loop $\rho_2$. 
By tracing the deformation of $\rho_3$, we can show that 
$\rho_3$ is changed into a loop turning around once the point $P_2$ positively.
Since the base point $\dot x$ is moved along $\rho_2$, 
this loop is equal to $\rho_2\cdot \rho_3\cdot\rho_2^{-1}$. 
Similarly,  we can show that 
$\rho_2^2\cdot \rho_3\cdot\rho_2^{-2}$ is a loop turning around once 
the point $P_3$ positively.
Thus the loop $\rho_3\cdot (\rho_2\cdot \rho_3\cdot\rho_2^{-1})\cdot
(\rho_2^2\cdot \rho_3\cdot\rho_2^{-2})$ turns around once the all points 
$P_1,P_2$ and $P_3$ positively. 
If we deform the line $L$ to  $x_2=0$, then $P_1$, $P_2$ and $P_3$ 
confluent to the point $(1,0)$. 
Hence  we can see that this loop is homotopic to $\rho_2'$.
It is clear that $\rho_2'$ commutes with $\rho_2$ since 
$\rho_2'$ is in the line $x_2=0$.  
\end{proof}

\begin{remark}
\label{rem:Artin}
The relation
$$\rho_2\cdot [
\rho_3 (\rho_2 \rho_3 \rho_2^{-1})
(\rho_2^2 \rho_3\rho_2^{-2})]=
[\rho_3 (\rho_2 \rho_3 \rho_2^{-1})
(\rho_2^2 \rho_3\rho_2^{-2})]\cdot \rho_2
$$
is equivalent to 
$$(\rho_2\rho_3)^3=(\rho_3\rho_2)^3$$
as elements of $\pi_1(X,\dot x)$. 
We also have 
$$(\rho_1\rho_3)^3=(\rho_3\rho_1)^3.$$
An Artin group of infinite type is defined by three generators  
$\varrho_1$, $\varrho_2$ and 
$\varrho_3$ together with three relations among them:
$$
\varrho_2\varrho_1=\varrho_1\varrho_2,\quad 
(\varrho_3\varrho_1)^3=(\varrho_1\varrho_3)^3,\quad 
(\varrho_3\varrho_2)^3=(\varrho_2\varrho_3)^3.
$$
\end{remark}

\section{Monodromy representation}
\label{sec:Monodromy}
In this section, we assume that 
\begin{equation}
\label{eq:non-integral}
\begin{matrix}
a_i,\  
\ b_j,\ b_k,\ b_1-b_2,\ b_3-b_4,\ 
a_i-b_j,\ a_i-b_k, \\
a_i-b_j-b_k,\ a_1+a_2+a_3-b_1-b_2-b_3-b_4
\end{matrix}
\notin \Z
\end{equation}
for $i=1,2,3$ and $j=1,2$ and $k=3,4.$ 
Note that the condition (\ref{eq:non-integral}) is stable under the 
action of $S_3\times D_4$.  

We array the fundamental system of solutions 
to $\hge{a}{B}$ on $U$ in Theorem \ref{th:fund-sols} 
as a column vector  
$$\mathbf{F}^g(x)
=\tr\big(g_{00}F_{00}(x),\dots,
g_{jk}F_{jk}(x), \dots \big)$$
in the order (\ref{eq:order}), 
where 
$g_{jk}$ are non-zero constants.
Let $M_i^g$ $(i=1,2,3)$ be the circuit matrix along $\rho_i$  with respect to 
this vector valued function. That is,  
this vector valued function is transformed into 
$$M_i^g \mathbf{F}^g(x)$$
by the analytic continuation along the loop $\rho_i$. 

\begin{lemma}
\label{lem:M1M2}
We have 
\begin{align*}
M_1^g&=\diag(1,\b_1^{-1},\b_2^{-1},1,\b_1^{-1},\b_2^{-1},
1,\b_1^{-1},\b_2^{-1}),\\
M_2^g&=\diag(1,1,1,\b_3^{-1},\b_3^{-1},\b_3^{-1},
\b_4^{-1},\b_4^{-1},\b_4^{-1}),
\end{align*}
for any $g=(g_{00},g_{10},\dots,g_{22})$, where $\b_j=e^{2\pi\sqrt{-1}b_j}$ 
$(j=1,\dots,4)$ and 
$\diag(c_1,\dots,c_m)$ denotes the diagonal matrix of size $m$ 
with diagonal entries $c_1,\dots,c_m$.

\end{lemma}
\begin{proof}
It is a direct consequences of Theorem \ref{th:fund-sols} that 
the circuit matrices $M_1$ and $M_2$ take the forms in this lemma.
\end{proof}

The essential problem is the determination of the circuit matrix $M^g_3$.
At first, we study eigenspaces of $M^g_3$.
\begin{lemma}
\label{lem:eigen-spaces}
Suppose that the circuit matrix $M^g_3$ is diagonalizable.
Then its eigenvalues are $1$ and $\l(\ne1)$. 
The $\l$-eigenspace is $1$ dimensional and 
the $1$-eigenspace is $8$ dimensional.
\end{lemma}

\begin{proof}
We construct a vanishing cycle $\OO$ when $x=(x_1,x_2)\in X$ 
approaches to $P_1=(x'_1,x'_2)\in L\cap S$ along a part of the loop $\rho_3$.
The integral $\int_{\OO}u(t,x)dt$ becomes a $\l$-eigenvector.

We regard the space $T$ in (\ref{eq:int-space}) as a fiber bundle 
over the base $(t_1,t_3)$-space with the fiber $(t_2,t_4)$-space.
At first, we fix $x$ and a generic $(t_1,t_3)$. 
Let $\OO_{24}$ be the real $2$-dimensional chamber surrounded by $\ell_{24}(t)=0$ 
and $q(t)=0$ in the $(t_2,t_4)$-space and including $\triangle$ 
in Figure \ref{fig:cycle}. 
By moving $(t_1,t_3)$ in a $2$-chain, 
we construct a locally finite $4$-chain as a family of $\OO_{24}$ 
over $(t_1,t_3)$.
Note that $\ell_{24}(t)=0$ is tangent to $q(t)=0$ if and only if 
$R_2(x_1/t_1,x_2/t_3)=0$, where 
$$R_2(y_1,y_2)=(1-y_1-y_2)^2-4y_1y_2=y_1^2+y_2^2+1-2y_1y_2-2y_1-2y_2.$$
Since $t_1\ne 0$ and $t_3\ne 0$ in the base space, the 
condition $R_2(x_1/t_1,x_2/t_3)=0$ is equivalent to 
$$q_4(t_1,t_3,x)=(t_1t_3-x_1t_3-x_2t_1)^2-4x_1x_2t_1t_3=0.$$
Note that the chamber $\OO_{24}$ vanishes  
on this curve for the fixed $(x_1,x_2)$ in the base space $(t_1,t_3)$.  
It is easy to see that this quartic curve in $\C^2$ has 
a cusp singular point at $(t_1,t_3)=(0,0)$ for any $x\in X$ and that  
this curve intersects the line $t_i=0$ $(i=1,3)$ 
only at $(0,0)$ in $\C^2$.
We consider the intersection of $q_4(t_1,t_3,x)=0$ and line 
$L_{13}:t_1+t_3=1$.
By eliminating $t_3$ from $q_4(t_1,t_3,x)=0$ and $t_1+t_3=1$, 
we have 
\begin{align*}
&
t_1^4-2(x_1-x_2+1)t_1^3+(x_1^2+x_2^2+1+2x_1x_2+4x_1-2x_2)t_1^2
\\
&-2x_1(x_1+x_2+1)t_1+x_1^2=0.
\end{align*}
Since its discriminant with respect to $t_1$ is 
$$256x_1^3x_2^3R(x_1,x_2),$$
$q_4(t_1,t_3,x)=0$ and $t_1+t_3=1$  
intersect at four distinct points for any $x\in X$; see Figure \ref{fig:Q4}.
\begin{figure}
\begin{center}
\includegraphics[width=10cm]{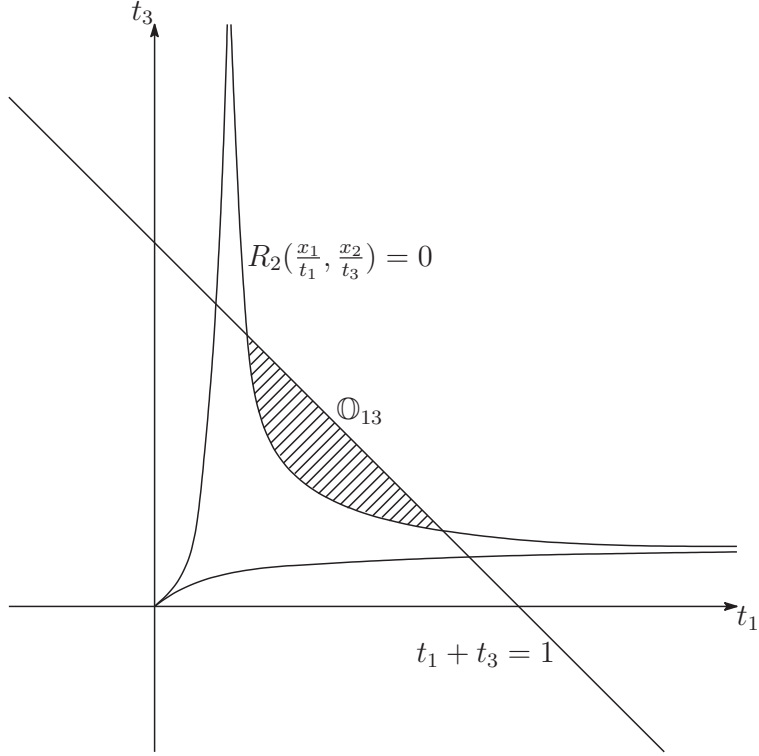}
\end{center}
\caption{$q_4(t_1,t_3,x)=0$ in $(t_1,t_3)$ space}
\label{fig:Q4}
\end{figure}

We set 
\begin{equation}
\label{eq:vanishing}
\OO=\bigcup_{(t_1,t_3)\in \OO_{13}} \OO_{24},
\end{equation}
where the region $\OO_{13}$ in the $(t_1,t_3)$-space 
is surrounded by the line $t_1+t_3=0$ and $q_{4}(t_1,t_3,x)=0$, 
see Figure \ref{fig:Q4}. 
We claim that this is a cycle as a locally finite chain.
Its boundary consists of 
$$\bigcup_{(t_1,t_3)\in \pa \OO_{13}} \OO_{24},\quad 
\bigcup_{(t_1,t_3)\in \OO_{13}} \pa\OO_{24}.$$
Since $\OO_{24}$ is a locally finite cycle, $\pa\OO_{24}=0$. 
The boundary component of $\OO_{13}$ in $t_1+t_3=1$ 
vanishes as a locally finite chain. 
Since $\OO_{24}$ vanishes over the boundary component of $\OO_{13}$
in $q_4(t_1,t_3,x)=0$, 
$\bigcup_{(t_1,t_3)\in \pa \OO_{13}} \OO_{24}=0$.
Hence $\OO$ is a locally finite cycle.

We consider the limit as $x$ to $P_1=(x'_1,x'_2)$. 
If $x=x'$ then the line $t_1+t_3=1$ tangents to $q_4(t_1,t_3,x')=0$ 
and we can show that $\OO_{13}$ vanishes; see Figure \ref{fig:deg-Q4}.
Hence $\OO$ is a required cycle.
\begin{figure}
\begin{center}
\includegraphics[width=10cm]{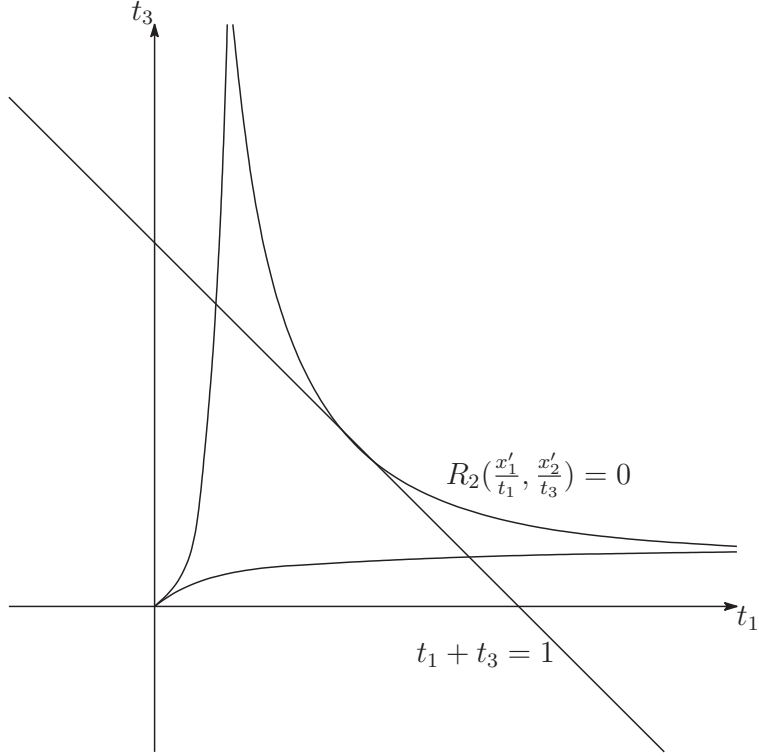}
\end{center}
\caption{$q_4(t_1,t_3,x')=0$ in $(t_1,t_3)$ space}
\label{fig:deg-Q4}
\end{figure}

By tracing the movement of $\OO$  along the loop $\rho_3$, we see that 
the deformed $\OO$ coincides with the initial $\OO$. 
Since the branch of $u(t,x)$ changes by the continuation along $\rho_3$, 
the integral $\int_{\OO}u(t,x)dt$  is multiplied $\l\ne 1$ 
under suitable non-integrable conditions on the parameters $a$ and $B$.
Thus the circuit matrix $M^g_3$ has an eigenvalue $\l$ different from $1$. 

Finally, we show that the $1$-eigenspace of $M^g_3$ is $8$ dimensional.
Let $\mathrm{Sol}(U_{P_1})$ be the vector space of 
holomorphic solutions to $\hge{a}{B}$, 
which can be extended to a neighborhood $U_{P_1}$ of 
$P_1=(x'_1,x'_2)$ in $\C^2$. 
It is sufficient to show that $\mathrm{Sol}(U_{P_1})$ is $8$ dimensional. 
By following the proof of Corollary \ref{cor:Euler}, 
we can show that $\dim \mathrm{Sol}(U_{P_1})$ 
is equal to the Euler number $\chi(T')$ of the space $T'$ for 
$x=(x'_1,x'_2)$. 
On the other hand, we have $\chi(T)=9$ for any $x\in X$ 
by Corollary \ref{cor:Euler}.  
Since the space $T'$ loses a $4$-chain $\OO$ homeomorphic to $\R^4$ 
from the space $T$, we have $\chi(T')=9-1=8$.
\end{proof}

\begin{remark}
\label{rem:doubly}
For a generic point $x$ of $R(x)=0$, 
$q_4(t_1,1-t_1,x)=0$ has a double root and two simple roots.
At the node $(x_1,x_2)=(-1,-1)$ of $R(x)=0$, 
$q_4(t_1,1-t_1,x)=0$ has two double roots, and 
$t_1+t_3=1$ becomes a bi-tangent of $q_4(t_1,t_3,x)=0$ with 
tangent points $(t_1,t_3)=(-\w,-\w^2),(-\w^2,-\w)$. 
For $x=(1,0)$ or $(0,1)$, $q_4(t_1,t_3,x)=0$ degenerates to 
a product of duplicate lines $(t_1-1)^2t_3^2$ or $t_1^2(t_3-1)^2$, 
and it intersects  the line $t_1+t_3=1$ at 
the quadruple point $(t_1,t_3)=(1,0)$ or $(0,1)$, respectively. 
\end{remark}

Next we normalize $g$ so that $M^g_3$ admits a simple expression.
For this purpose, 
we introduce the intersection pairing $\cI$ between 
the twisted homology groups 
$H_4(T,\mathcal{L}_u)$ and $H_4^{lf}(T,\mathcal{L}_u^\vee)$, where 
$\mathcal{L}_u^\vee$ is the dual local system of $\mathcal{L}_u$.
Here $\ ^\vee$  denotes 
the sign change operator 
$$
\nu(a_1,a_2,a_3;b_1,b_2,b_3,b_4)^\vee=\nu(-a_1,-a_2,-a_3;-b_1,-b_2,-b_3,-b_4)
$$
for any function $\nu$ of parameters. For example, we have 
$\a_i^\vee=1/\a_i$ and $u(t,x)^\vee=1/u(t,x)$. 
Under our assumption (\ref{eq:non-integral}), 
the natural map from $H_4(T,\mathcal{L}_u)$ to 
$H_4^{lf}(T,\mathcal{L}_u^\vee)$ is isomorphic, and 
its inverse $\reg$ is defined. 
We can regard the intersection pairing as defined between 
$H_4(T,\mathcal{L}_u)$ and $H_4(T,\mathcal{L}_u^\vee)$.
There exist twisted cycles $\De_{jk}^{g,u}$ 
given by 4-chain $\De_{jk}^g$ with a branch of $(u,t)$
such that 
$$\int_{\De_{jk}^g}u(t,x)dt=g_{jk}F_{jk}(x)$$ 
for any $0\le j,k\le 2$.
We define a $9\times 9$ matrix $H^g$ by 
$$H^g=\frac{1}{\cI(\De_{00}^{g,u},{\De_{00}^{g,u}}^\vee)}
\left(\cI(\De_{jk}^{g,u},{\De_{j'k'}^{g,u}}^\vee)\right)_{(jk),(j'k')},
$$
where $(jk)$ and $(j'k')$ are arrayed 
in the order (\ref{eq:order}).
We compute $\cI(\De_{jk}^{u},{\De_{j'k'}^{u}}^\vee)$ in 
\S\ref{sec:Intersection}.
In this section, we treat the entries $h_{10},h_{20},\dots, h_{22}$ of $H$ 
as indeterminate for the moment. 

\begin{lemma}
\label{lem:H-exist}
The matrix $H^g$  satisfies
$$M_i^g H^g \tr(M_i^g)^\vee=H^g\quad (i=1,2,3),$$
and takes the form $\diag(1,h^g_{10},\dots,h^g_{22})$.
\end{lemma}
\begin{proof}
By the local triviality of the intersection form, 
we can show that the relations 
is satisfied by any circuit matrices $M_\rho^g$ along loops 
$\rho$ with terminal $\dot x$; 
refer to  the proof of Lemma 4 in \cite{MY} for details.
By the property 
$$M_i^g H^g \tr(M_i^g)^\vee=H^g\quad (i=1,2),$$
for $M_1$ and $M_2$ given in Lemma \ref{lem:M1M2}, 
$H$ should be diagonal under the assumption (\ref{eq:non-integral}).
It is clear that its top-left entry is $1$.
\end{proof}

\begin{lemma}
\label{lem:M3}
Suppose that the eigenvalue $\l$ of $M^g_3$ is different from $1$ 
and that $v=(v_{00},\dots,v_{22})$ is a $\l$-eigenvector of $M^g_3$. 
\begin{enumerate}
\item The eigenspace of $M_3^g$ of eigenvalue $1$ is characterized as 
$$\{w\in \C^9\mid w H^g \tr v^\vee =0\}.$$
\item The vector $v$ satisfies $vH^g \tr v^\vee\ne0$.
\item 
The circuit matrix $M^g_3$ is expressed as
$$M^g_3=I_9-\frac{(1-\l)}{v H^g\tr v^\vee }H^g\tr v^\vee v.$$
\item 
No entries of $v$ vanish.
\end{enumerate}
\end{lemma}

\begin{proof}
(1) Let $w$ be a $1$-eigenvector of $M^g_3$. Then we have
$$w H^g \tr v^\vee =w [M^g_3 H^g \tr(M^g_3)^\vee ]\tr v^\vee =
(w M^g_3) H^g \tr(v M^g_3)^\vee 
= \l^\vee w H^g \tr v^\vee,$$
$$(1-\l^\vee) w H^g \tr v^\vee=0.$$
Since $(\l^\vee)^\vee=\l$ and $\l\ne 1$,  
the factor $1-\l^\vee$ does not vanish. 
Hence we have $w H^g \tr v^\vee=0$.

\noindent(2) It is known that 
the intersection form is a perfect pairing between twisted homology groups. 
Since the matrix $H$ corresponds to the intersection matrix,   
it is non-degenerate. 
If $vH^g \tr v^\vee=0$ then $u H^g \tr v^\vee=0$ for any $u\in \C^9$ 
by the result (1) and Lemma \ref{lem:eigen-spaces}.
This means that $H$ degenerates. 

\noindent(3) Put 
$$M'=I_9-\frac{1-\l}{v H^g\tr v^\vee }H^g\tr v^\vee v.$$
Since we have
$$
v M'=v+ \frac{1-\l}{v H^g\tr v^\vee }v{H^g}\! \tr v^\vee v=\l v,\ 
w M'=w+ \frac{1-\l}{v H^g\tr v^\vee }w{H^g}\! \tr v^\vee v=w
$$
for any $w$ in $\{w\in \C^9\mid w H^g \tr v^\vee =0\}$, 
the eigenvalues and eigenspaces of $M'$ and $M^g_3$ coincide.
Hence $M'$ is equal to $M^g_3.$

\noindent(4) 
Assume that an entry $v_i$ of $v$ vanishes. 
Then the unit vector $\ex_i$ is a $1$-eigenvector of $M^g_3$ since 
$\ex_i H \tr v^\vee=0$. Moreover, $\ex_i$ is an eigenvector of 
$M^g_1$ and that of $M^g_2$.
Hence the space spanned by $\ex_i$ is invariant subspace of 
the monodromy representation. 
By removing the power function $x_1^{1-b_j}x_2^{1-b_k}$ 
from corresponding solution to $\ex_i$ and 
restricting it to $x_2=0$, we have the differential equation 
$_3\cF_2$ associated with it. This has an invariant subspace of 
the monodromy representation. By Proposition 3.3 in \cite{BH}, 
some of $a_i$, $a_i-b_1$, $a_i-b_2$ belong to $\Z$. 
It contradicts to our assumption (\ref{eq:non-integral}).
\end{proof}

We normalize $g=(g_{00},g_{10},\dots, g_{22})$ so that 
the $\l$-eigenvector $v$ of $M^g_3$ 
becomes $\one=(1,\dots,1)$. From now on, we fix $g$, and 
we use symbols  
$\mathbf{F}(x)$, 
$H=(1,h_{10},\dots,h_{22})$, $M_i$ $(i=1,2,3)$ for this fixed $g$.
Note that $M_1$ and $M_2$ are given in Lemma \ref{lem:M1M2} and that
$$M_3=I_9-\frac{1-\l}{\one H\tr\one}H\tr \one\one,$$
where we regard $\l$ and entries of $H$ as indeterminate.

Finally, we determine them by considering the restrictions $\mathbf{F}(x)$ 
to $x_i=0$ ($i=1,2$).

\begin{proposition}
\label{prop:H-entry}
The eigenvalue $\l$ and the diagonal entries of 
$H$ are 
\begin{align*}
\l&=\frac{\b_1\b_2\b_3\b_4}{\a_1\a_2\a_3},\\
h_{10}&= \frac{(\a_1-\b_1)(\a_2-\b_1)(\a_3-\b_1)(\b_2-1)}
{(\a_1-1)(\a_2-1)(\a_3-1)\b_1(\b_1-\b_2)},\\
h_{20}&= \frac{(\a_1-\b_2)(\a_2-\b_2)(\a_3-\b_2)(\b_1-1)}
{(\a_1-1)(\a_2-1)(\a_3-1)\b_2(\b_2-\b_1)},
\end{align*}
\begin{align*}
h_{01}&= \frac{(\a_1-\b_3)(\a_2-\b_3)(\a_3-\b_3)(\b_4-1)}
{(\a_1-1)(\a_2-1)(\a_3-1)\b_3(\b_3-\b_4)},\\
h_{11}&= \frac{(\a_1-\b_1\b_3)(\a_2-\b_1\b_3)(\a_3-\b_1\b_3)(\b_2-1)(\b_4-1)}
{(\a_1-1)(\a_2-1)(\a_3-1)\b_1\b_3(\b_1-\b_2)(\b_3-\b_4)},\\
h_{21}&= \frac{(\a_1-\b_2\b_3)(\a_2-\b_2\b_3)(\a_3-\b_2\b_3)(\b_1-1)(\b_4-1)}
{(\a_1-1)(\a_2-1)(\a_3-1)\b_2\b_3(\b_2-\b_1)(\b_3-\b_4)},
\end{align*}
\begin{align*}
h_{02}&= \frac{(\a_1-\b_4)(\a_2-\b_4)(\a_3-\b_4)(\b_3-1)}
{(\a_1-1)(\a_2-1)(\a_3-1)\b_4(\b_4-\b_3)},\\
h_{12}&= \frac{(\a_1-\b_1\b_4)(\a_2-\b_1\b_4)(\a_3-\b_1\b_4)(\b_2-1)(\b_3-1)}
{(\a_1-1)(\a_2-1)(\a_3-1)\b_1\b_4(\b_1-\b_2)(\b_4-\b_3)},\\
h_{22}&= \frac{(\a_1-\b_2\b_4)(\a_2-\b_2\b_4)(\a_3-\b_2\b_4)(\b_1-1)(\b_3-1)}
{(\a_1-1)(\a_2-1)(\a_3-1)\b_2\b_4(\b_2-\b_1)(\b_4-\b_3)}.
\end{align*}
\end{proposition}

\begin{remark}
\label{rem:S3D4onH} The group $S_3\times D_4$ naturally acts on the entries 
$h_{jk}$ of $H$.  This is compatible with that on 
the fundamental solutions $F_{jk}(x)$.
\end{remark}
\begin{proof}
Consider the loop $\rho_3\cdot(\rho_2\cdot \rho_3\cdot \rho_2^{-1})
\cdot (\rho_2^2\cdot \rho_3\cdot \rho_2^{-2})$.
By Proposition \ref{prop:reduction}, it commutes with $\rho_2$ and 
homotopic to a loop $\rho_2'$ turning $x_1=1$ once positively 
in the space $\{(x_1,0)\in \C^2\mid x_1\ne0,1\}$.
Let $N_2$ be the circuit matrix  along this loop with respect to 
$\mathbf{F}(x)$.
Since it commutes with $M_2$, it is a block diagonal matrix 
with respect to the $(3,3,3)$-partition. 
Since the restrictions of 
$F_{j0}(x)$ $(j=0,1,2)$ 
to $x_2=0$ reduce to a fundamental system $_3\cF_2\begin{pmatrix}
a_1,a_2,a_3\\ b_1,b_2
\end{pmatrix}$, the top-left block of $N_2$ coincides with 
the circuit matrix of $\rho_2'$  with respect to their restrictions.
This $3\times3$ circuit matrix is expressed in Proposition 3.1 of \cite{M2} as
$$I_3-\frac{1-\l'}{\one' H' \tr \one'} H'\tr \one'\one',$$
where 
$I_3$ is the unit matrix of size $3$, $\l'=\dfrac{\b_1\b_2}{\a_1\a_2\a_3}$, 
$\one'=(1,1,1)$ and $H'=\diag(1,h_{10},h_{20})$ for 
given $h_{10}$ and $h_{20}$ in this proposition.
By following the proof of Proposition 4.1 of \cite{M2},  
we can show the coincidence of two fundamental systems. 
By the uniqueness of the matrix $H$, we determine the values of 
$h_{10}$ and $h_{20}$ as in this proposition.

We consider the middle block of $N_2$ and the ratios of the solutions
$F_{j1}(x)$ $(j=0,1,2)$. 
By taking out the factor $x_2^{1-b_3}$ from them, 
we have the circuit matrix of $\rho_2'$ 
with respect to a fundamental system of  
$_3\cF_2\begin{pmatrix}
a_1-b_3+1,a_2-b_3+1,a_3-b_3+1\\ b_1,\qquad b_2
\end{pmatrix}$. It is 
$$I_3-\frac{1-\l'}{\one' H' \tr \one'} H'\tr \one'\one',$$
where 
$\l'=\dfrac{\a_1\a_2\a_3}{\b_1\b_2\b_3^3}$, 
$H'=\diag(1,h'_{10},h'_{20})$ and $h_{j0}'$ is 
the replacement of $h_{j0}$ by 
$(\a_1,\a_2,\a_3)\to(\a_1/\b_3,\a_2/\b_3,\a_3/\b_3)$. 
By the uniqueness of $H$, we have 
\begin{align*}
\dfrac{h_{11}}{h_{01}}&=
\frac{(\a_1/\b_3-\b_1)(\a_2/\b_3-\b_1)(\a_3/\b_3-\b_1)(\b_2-1)}
{(\a_1/\b_3-1)(\a_2/\b_3-1)(\a_3/\b_3-1)\b_1(\b_1-\b_2)},\\
\dfrac{h_{21}}{h_{01}}
&= \frac{(\a_1/\b_3-\b_2)(\a_2/\b_3-\b_2)(\a_3/\b_3-\b_2)(\b_1-1)}
{(\a_1/\b_3-1)(\a_2/\b_3-1)(\a_3/\b_3-1)\b_2(\b_2-\b_1)},\\
h_{11}&=h_{01}\frac{(\a_1-\b_1\b_3)(\a_2-\b_1\b_3)(\a_3-\b_1\b_3)(\b_2-1)}
{(\a_1-\b_3)(\a_2-\b_3)(\a_3-\b_3)\b_1(\b_1-\b_2)},\\
h_{21}&=h_{01} 
\frac{(\a_1-\b_2\b_3)(\a_2-\b_2\b_3)(\a_3-\b_2\b_3)(\b_1-1)}
{(\a_1-\b_3)(\a_2-\b_3)(\a_3-\b_3)\b_2(\b_2-\b_1)}.
\end{align*}
By the bottom right block of $N_2$, we have
\begin{align*}
h_{12}&=h_{02}\frac{(\a_1-\b_1\b_4)(\a_2-\b_1\b_4)(\a_3-\b_1\b_4)(\b_2-1)}
{(\a_1-\b_4)(\a_2-\b_4)(\a_3-\b_4)\b_1(\b_1-\b_2)},\\
h_{22}&=h_{02} 
\frac{(\a_1-\b_2\b_4)(\a_2-\b_2\b_4)(\a_3-\b_2\b_4)(\b_1-1)}
{(\a_1-\b_4)(\a_2-\b_4)(\a_3-\b_4)\b_2(\b_2-\b_1)}.
\end{align*}

By considering the loop $\rho_3\cdot(\rho_1\cdot \rho_3\cdot \rho_1^{-1})
\cdot (\rho_1^2\cdot \rho_3\cdot \rho_1^{-2})$, 
we determine the values of 
$h_{01}$ and $h_{02}$ as in this proposition. 
Hence we have the entries of $H$.

We compute the determinant of $N_2$. Since $\det(M_3)=\l$, we have 
$\det(N_2)= \l^3.$
On the other hand, the determinants of block matrices of $N_2$ are 
$$\frac{\b_1\b_2}{\a_1\a_2\a_3},\quad 
\frac{\b_1\b_2\b_3^3}{\a_1\a_2\a_3},\quad 
\frac{\b_1\b_2\b_4^3}{\a_1\a_2\a_3}.
$$
Thus we have 
$$\l^3=\Big(\frac{\b_1\b_2\b_3\b_4}{\a_1\a_2\a_3}\Big)^3.$$
Note that $(1,1,1,0,\dots,0)$ is an eigenvector of $N_2$. 
This property yields that $\l$ is as in this proposition. 
\end{proof}
We conclude results in this section as the following theorem. 
\begin{theorem}
\label{th:monod}
Suppose the condition (\ref{eq:non-integral}). Then we have 
\begin{align*}
M_1&=\diag(1,\b_1^{-1},\b_2^{-1},1,\b_1^{-1},\b_2^{-1},
1,\b_1^{-1},\b_2^{-1}),\\
M_2&=\diag(1,1,1,\b_3^{-1},\b_3^{-1},\b_3^{-1},
\b_4^{-1},\b_4^{-1},\b_4^{-1}),\\
M_3&=I_9-\frac{1-\l}{\one H\tr\one}H\tr \one\one,
\end{align*}
where $I_9$ is the unit matrix of size $9$, $\a_i=\exp(2\pi\sqrt{-1}a_i)$, 
$\b_j=\exp(2\pi\sqrt{-1}b_j)$, $\one=(1,\dots,1)\in \N^9$, 
$\l=(\b_1\b_2\b_3\b_4)/(\a_1\a_2\a_3)$, 
$H=\diag(1,h_{10},h_{20},\dots,h_{22})$ whose entries are 
given in Proposition \ref{prop:H-entry}. 
\end{theorem}

\begin{cor}
\label{cor:regular-singular}
The system $\hge{a}{B}$ is regular singular.  
\end{cor}
\begin{proof}
We have only to consider the behavior of its solutions
around each component of the singular locus $S$.
\end{proof}

\begin{remark}
\begin{enumerate}
\item 
To study the monodromy of $_3\cF_2\begin{pmatrix}
a_1,a_2,a_3\\ b_1,b_2
\end{pmatrix}$ in \cite{M2}, 
we assume 
$$a_i-a_{i'}\notin \Z,\quad (1\le i<i'\le 3)$$
so that the eigen-polynomial $f(t)$ of a circuit matrix implies 
two independent linear equations by the substitution of $t=\a_i$ $(i=1,2,3)$.
Since we can remove these conditions by considering $f'(t)$ and $f''(t)$ 
in the case of $\a_i=\a_{i'}$, we do not need these conditions 
for Theorem \ref{th:monod}. 
\item 
Note that we can cancel the factor $\a_1\a_2\a_3-\b_1\b_2\b_3\b_4$ 
from $\dfrac{1-\l}{\one H\tr \one}$ in the expression of $M_3$ 
in Theorem \ref{th:monod}. 
Thus 
we can remove the condition 
$$a_1+a_2+a_3-b_1-b_2-b_3-b_4\notin \Z$$
from (\ref{eq:non-integral}).
\end{enumerate}
\end{remark}

\section{Intersection numbers}
\label{sec:Intersection}
In this section, we compute the intersection numbers of twisted cycles 
corresponding to fundamental solutions.
We start from the following lemma, which is a consequence of 
Lemma \ref{lem:H-exist}. 
\begin{lemma}
\label{lem:orth}
If $(j,k)\ne(j',k')$ for $0\le j,j',k,k'\le 2$ then 
$$\cI(\De_{jk}^u,{\De_{j'k'}^u}^\vee)=0,$$
where $\De_{jk}^u$ are given in Remarks \ref{rem:regularization}, 
\ref{rem:Delta10} ,\ref{rem:Delta11} and \ref{rem:Delta21}, and 
${\De_{jk}^u}^\vee$ denotes the image of $\De_{jk}^u$ under the map 
$\;^\vee$.
\end{lemma}

We compute intersection numbers by the two-dimensional reduction technique as 
follows. 
\begin{lemma}
\label{lem:int-SFC}
The intersection numbers $\cI(\De_{00}^u,{\De_{00}^u}^\vee)$, 
$\cI(\De_{10}^u,{\De_{10}^u}^\vee)$, $\cI(\De_{11}^u,{\De_{11}^u}^\vee) $
and $\cI(\De_{21}^u,{\De_{21}^u}^\vee)$ are 
\begin{align*}
&\frac{(\a_1-1)\b_1\b_3}
{(1-\b_1)(1-\b_3)(\a_1-\b_1\b_3)}
\cdot 
\frac{(\a_2-1)\b_2\b_4}
{(1-\b_2)(1-\b_4)(\a_2-\b_2\b_4)},\\[2mm]
&\dfrac{(\a_1-\b_1)(\a_3-\b_1)\b_3}
{(\a_1-\b_1\b_3)(\a_3-1)(1-\b_1)(1-\b_3)}
\cdot
\dfrac{(\a_2-\b_1)\b_2\b_4}
{(\a_2-\b_2\b_4)(\b_2-\b_1)(1-\b_4)},\\[2mm]
&\frac{\a_3-\b_1\b_3}
{(1-\b_1)(1-\b_3)(\a_3-1)}
\cdot 
\frac{(\a_2-\b_1\b_3)\b_2\b_4}
{(\b_2-\b_1)(\b_4-\b_3)(\a_2-\b_2\b_4)},\\[2mm]
&\dfrac{(\a_1-\b_2\b_3)(\a_2-\b_2\b_3)\b_1\b_4}
{(\a_1-\b_1\b_3)(\a_2-\b_2\b_4)(\b_1-\b_2)(\b_4-\b_3)}\cdot
\dfrac{\a_3-\b_2\b_3}
{(\a_3-1)(1-\b_2)(1-\b_3)},
\end{align*}
respectively.
\end{lemma}
\begin{proof}
By 
Remark \ref{rem:regularization} (2), 
we can compute $\cI(\De_{00}^u,{\De_{00}^u}^\vee)$ as the product of the 
intersection numbers $\cI(\triangle_1^{u_1}, {\triangle_1^{u_1}}^\vee)$ and  
$\cI(\triangle_2^{u_2}, {\triangle_2^{u_2}}^\vee)$, where 
$\triangle_i^{u_i}$ $(i=1,2)$ is a twisted cycle given by $\triangle_i$
and $u_i$ in the $(t_i,t_{i+2})$-space. 
By using results in 
\S 3.1 of Chapter VIII in \cite{Y}, 
we have
\begin{align*}
\cI(\triangle_1^{u_{1}},{\triangle_1^{u_{1}}}^\vee)&=
\frac{(1-\a_1^{-1})}
{(1-\b_1^{-1})(1-\b_3^{-1})(1-\a_1^{-1}\b_1\b_3)}
,\\
\cI(\triangle_2^{u_{2}},{\triangle_2^{u_{2}}}^\vee)&
=\frac{(1-\a_2^{-1})}
{(1-\b_2^{-1})(1-\b_4^{-1})(1-\a_2^{-1}\b_2\b_4)}.
\end{align*}

The intersection number $\cI(\De_{10}^u,{\De_{10}^u}^\vee)$ can be 
computed in the $s$-space by the involution $\imath_{13}$. 
It reduces to the product of 
$\cI(\square_1^{u'_1}, {\square_1^{u'_1}}^\vee)$ and  
$\cI(\triangle_2^{u'_2}, {\triangle_2^{u'_2}}^\vee)$, where 
$\square_1^{u_1'}$ and $\triangle_2^{u'_2}$ are twisted cycles 
in the $(s_i,s_{i+2})$-space. 
We can similarly compute them as 
\begin{align*}
\cI(\square_{1}^{u'_{1}},{\square_{1}^{u'_{1}}}^\vee)&=
\dfrac{1-\a_3^{-1}\b_1}
 {(1-\b_1)(1-\a_3^{-1})}\cdot 
 \dfrac{1-\a_1^{-1}\b_1}
 {(1-\b_3^{-1})(1-\a_1^{-1}\b_1\b_3)}
,\\
\cI(\triangle_2^{u'_{2}},{\triangle_2^{u'_{2}}}^\vee)&=
\dfrac{1-\a_2^{-1}\b_1}
{(1-\b_1\b_2^{-1})(1-\b_4^{-1})(1-\a_2^{-1}\b_2\b_4)}.
\end{align*}
 
The rests also can be similarly computed. 
\end{proof}

\begin{proposition}
\label{prop:proportional}
We have 
$$\cI\left(\De_{jk}^u,{\De_{jk}^u}^\vee\right)
=h_{jk}\cdot \cI(\De_{00}^u,{\De_{00}^u}^\vee),
$$
for any $0\le j,k\le 2$, 
where 
$h_{jk}$ are given in Proposition \ref{prop:H-entry}.
\end{proposition}

\begin{proof}
Lemma \ref{lem:int-SFC} yields that this proposition holds 
for the twisted cycles $\De_{jk}^u$ 
($(jk)=(10)$, $(11)$, $(21)$). Recall that 
the other twisted cycles are constructed by these twisted cycles and 
some actions of $S_3\times D_4$ on the parameters.
By acting them on the obtained identities and using Remark \ref{rem:S3D4onH}, 
we have this proposition.  
For example, by the action of $((12),\sigma_{110})\in S_3\times D_4$
on the identity for $\De_{10}^u$, 
the intersection number 
$\cI\left(\De_{20}^u,{\De_{20}^u}^\vee\right)$  can be computed as 
$$
\dfrac{(\a_2-\b_2)(\a_3-\b_2)\b_4}
{(\a_2-\b_2\b_4)(\a_3-1)(1-\b_2)(1-\b_4)}
\cdot
\dfrac{(\a_1-\b_2)\b_1\b_3}
{(\a_1-\b_1\b_3)(\b_1-\b_2)(1-\b_3)},
$$
and it is equal to $h_{20} \cI\left(\De_{00}^u,{\De_{00}^u}^\vee\right)$,  
where $\sigma_{110}$ is given in Table \ref{tab:D4S4}.   
\end{proof}

\begin{remark}
\label{rem:int-action}
The intersection number 
$\cI\left(\De_{00}^u,{\De_{00}^u}^\vee\right)$ is not invariant under the 
action of $S_3\times D_4$ on the parameters. 
In the constructions of $\De_{jk}$ $(jk)=(20)$, $(01)$, $(02)$,
$(21)$ and $(12)$,  we use only actions which keep it invariant.
By Proposition \ref{prop:proportional} and Remark \ref{rem:S3D4onH}, 
the action of $S_3\times D_4$ on the ratio 
$$
\frac{\cI\left(\De_{jk}^u,{\De_{jk}^u}^\vee\right)}
{\cI\left(\De_{00}^u,{\De_{00}^u}^\vee\right)}
$$
is compatible with that on the fundamental solutions $F_{jk}(x)$.
This property yields that 
$$
\cI\left(\De_{\sigma(jk)}^u,{\De_{\sigma(jk)}^u}^\vee\right)=
\frac{\sigma\cdot \cI\left(\De_{jk}^u,{\De_{jk}^u}^\vee\right)}
{\sigma\cdot\cI\left(\De_{00}^u,{\De_{00}^u}^\vee\right)}
\times \cI\left(\De_{00}^u,{\De_{00}^u}^\vee\right),
$$
where $\sigma\in D_4$ and $\De_{\sigma(jk)}^u$ is the twisted cycle 
corresponding to the fundamental solution $\sigma\cdot F_{jk}$.
For example, $\cI\left(\De_{20}^u,{\De_{20}^u}^\vee\right)$ can be 
computed as 
$$
\frac{\sigma_{100}\cdot \cI\left(\De_{10}^u,{\De_{10}^u}^\vee\right)}
{\sigma_{100}\cdot\cI\left(\De_{00}^u,{\De_{00}^u}^\vee\right)}
\times \cI\left(\De_{00}^u,{\De_{00}^u}^\vee\right),
$$
where $\sigma_{100}$ is given in Table \ref{tab:D4S4}. 
\end{remark}

\begin{theorem}
The fundamental system $\mathbf{F}(x)$ is 
given by the integrals 
$$\mathbf{\Delta}(x)=\tr\big(\dots,\int_{\De_{jk}} u(t,x)dt,\dots\big),$$ 
where $(jk)$ are arrayed in the order $(00),(10),(20),\dots,(22)$. 
\end{theorem}
\begin{proof}
By Theorem \ref{th:int-rep} and Corollaries \ref{cor:2nd-int}, 
\ref{cor:5th-int} and \ref{cor:6th-int}, 
there exists a diagonal matrix $g=\diag(\dots,g_{jk},\dots)\in GL_9(\C)$ 
such that 
$$\mathbf{F}(x)=g\mathbf{\Delta}(x).$$
Moreover, it turns out by Proposition \ref{prop:proportional} that 
$$g H g^{\vee}=H,$$ 
which is  equivalent to $g_{jk}g_{jk}^\vee=1$ for any $(jk)$.
We show that $g$ is a scalar matrix.
By our normalization of $\mathbf{F}(x)$, it is sufficient 
to show that the $\l$-eigenvector 
of the circuit matrix of $\rho_3$ is $\one =(1,\dots,1)$ 
with respect to $\mathbf{\Delta}(x)$.
As in proved in Lemma \ref{lem:eigen-spaces}, 
the integral $\int_{\OO}u(t,x)dt$ corresponds to the $\l$-eigenvector of $M_3$.
We express $\OO^u$ as a linear combination 
\begin{equation}
\label{eq:lin-comb}
\OO^u=\sum_{0\le j,k\le 2} c_{jk}\De_{jk}^u,
\end{equation}
of $\De_{jk}^u\in H_4(T,\mathcal{L}_u)$, 
where $\OO^u$ is the twisted cycle defined by the $4$-cycle 
$\OO$ in (\ref{eq:vanishing}) and a branch of $u(t,x)$ on it.
By following the proof of Proposition 5.8 in \cite{G2},  
we can show that the key identity 
$$
\cI\left(\OO^u,{\De_{jk}^u}^\vee\right)=
\cI\left(\De_{jk}^u,{\De_{jk}^u}^\vee\right)
$$
for any $(jk)$. 
By considering the intersection numbers 
of both sides of   (\ref{eq:lin-comb}) and ${\De_{jk}^u}^\vee$,  
we see that $c_{jk}=1$ by the key identity and Lemma \ref{lem:orth}. 
Hence we have 
$$\OO^u = \sum_{0\le j,k\le 2} \De_{jk}^u
=(1,\dots,1)\;^t(\dots, \De_{jk}^u,\dots),$$
which shows the $\l$-eigenvector 
of the circuit matrix of $\rho_3$ is $\one$ 
with respect to $\mathbf{\Delta}(x)$ by the linearity of the integration. 
\end{proof}

Recall that the $(j,k)$-entry of $\mathbf{\Delta}(x)$ is 
$\int_{\De_{jk}}u(t,x)dt=F_{jk}(x)/C_{jk}$.  
Note that $C_{jk}$ times 
$$
\G(1-a_3)\G(b_1+b_3-a_1-1)\G(b_2+b_4-a_2-1)
$$
becomes 
 $$
C'_{jk}=
\frac{\prod_{i=1}^3 \G(b_{1j}+b_{2k}-a_i-1)}
{\prod_{1\le i\le 3}^{i\ne j} \G(b_{1j}-b_{1i})
 \prod_{1\le i\le 3}^{i\ne k} \G(b_{2k}-b_{2i})}.
$$
Thus we consider 
$$\frac{\mathbf{F}(x)}{\G(1-a_3)\G(b_1+b_3-a_1-1)\G(b_2+b_4-a_2-1)};$$
its $(j,k)$-entry ${G}_{jk}(x)$ is 
\begin{align}
\label{eq:N-series}
\frac{F_{jk}(x)}{C'_{jk}}=&\dfrac{\prod_{i=1}^3 \sin(\pi(a_i-b_{1j}-b_{2k}+2))}
{\prod_{1\le i\le 3}^{i\ne j} \sin(\pi(b_{1j}-b_{1i}))
 \prod_{1\le i\le 3}^{i\ne k} \sin(\pi(b_{2k}-b_{2i}))}
\\
\nonumber
&\cdot \sum_{n'_1,n'_2} \Big[\prod_{i=1}^3\frac{\G(a_i+n'_1+n'_2)}
{\G(b_{1i}+n'_1)\G(b_{2i}+n'_2)}\Big]
x_1^{n'_1}x_2^{n'_2},
\end{align}
where 
$n'_1$ and $n'_2$ run over the sets 
\begin{align*}
1-b_{1j}+\N&=\{1-b_{1j},2-b_{1j}, 3-b_{1j}, \dots\},\\
1-b_{2k}+\N&=\{1-b_{2k},2-b_{2k}, 3-b_{2k}, \dots\},
\end{align*}
respectively.
Note that $G_{jk}(x)$ is defined under conditions 
$$ 
b_{1i}-b_{1j}\ (i\in \{1,2,3\}-\{j\}),\quad 
b_{2i}-b_{1k}\ (i\in \{1,2,3\}-\{k\})\notin \Z.$$

\begin{remark}
\label{rem:irreducible}
In \cite{KMO1}, we give a linear transformation of them 
so that the entries are valid even in 
$$(\b_2-\b_1)(\b_2-1)(\b_1-1)(\b_4-\b_3)(\b_4-1)(\b_3-1)=0,$$ 
and study the irreducibility of the monodromy representation of $\hge{a}{B}$.
\end{remark}

We can regard the twisted homology group $H_4(T,\mathcal{L}_u)$ 
as the representation space of the monodromy of $\hge{a}{B}$. 
By the deformation of elements in $H_4(T,\mathcal{L}_u)$ along 
a loop $\rho$, we have a homomorphism 
$$\cM:\pi_1(X,\dot x)\ni \rho \mapsto 
\cM_\rho\in GL(H_4(T,\mathcal{L}_u)).$$
We express the circuit transformations $\cM_i=\cM_{\rho_i}$ along the loops 
$\rho_i$ ($i=1,2,3)$ in terms of the intersection form.

\begin{theorem}
\label{th:monod-int} 
We have 
\begin{align*}
&\cM_1(\De^u)=\\\De^u
&-(1-\b_1^{-1})
\left(\cI(\De^u,{\De^u_{10}}^\vee),\cI(\De^u,{\De^u_{11}}^\vee),
\cI(\De^u,{\De^u_{12}}^\vee)\right)
\cH_{11}^{-1} 
\begin{pmatrix}
\De^u_{10}\\
\De^u_{11}\\
\De^u_{12}\\
\end{pmatrix}\\
&-(1-\b_2^{-1})
\left(\cI(\De^u,{\De^u_{20}}^\vee),\cI(\De^u,{\De^u_{21}}^\vee),
\cI(\De^u,{\De^u_{22}}^\vee)\right)
\cH_{12}^{-1} 
\begin{pmatrix}
\De^u_{20}\\
\De^u_{21}\\
\De^u_{22}
\end{pmatrix},
\end{align*}
\begin{align*}
&\cM_2(\De^u)=\\ \De^u
&-(1-\b_3^{-1})
\left(\cI(\De^u,{\De^u_{01}}^\vee),\cI(\De^u,{\De^u_{11}}^\vee),
\cI(\De^u,{\De^u_{21}}^\vee)\right)
\cH_{21}^{-1} 
\begin{pmatrix}
\De^u_{01}\\
\De^u_{11}\\
\De^u_{21}\\
\end{pmatrix}\\
&-(1-\b_4^{-1})
\left(\cI(\De^u,{\De^u_{02}}^\vee),\cI(\De^u,{\De^u_{12}}^\vee),
\cI(\De^u,{\De^u_{22}}^\vee)\right)
\cH_{22}^{-1} 
\begin{pmatrix}
\De^u_{02}\\
\De^u_{12}\\
\De^u_{22}
\end{pmatrix},\\
&\cM_3(\De^u)=\De^u-\Big(1-\frac{\b_1\b_2\b_3\b_4}{\a_1\a_2\a_3}\Big)
\frac{\cI(\De^u,{\OO^u}^\vee)}{\cI(\OO^u,{\OO^u}^\vee)}\OO^u,
\end{align*}
where $\De^u\in H_4(T,\mathcal{L}_u)$ and 
\begin{align*}
\cH_{1j}&=\diag\left(\cI(\De^u_{j0},{\De^u_{j0}}^\vee),
\cI(\De^u_{j1},{\De^u_{j1}}^\vee),
\cI(\De^u_{j2},{\De^u_{j2}}^\vee)
\right),\\
\cH_{2j}&=\diag\left(\cI(\De^u_{0j},{\De^u_{0j}}^\vee),
\cI(\De^u_{1j},{\De^u_{1j}}^\vee),
\cI(\De^u_{2j},{\De^u_{2j}}^\vee)
\right).
\end{align*}
\end{theorem}

\begin{proof}
Let $\cM_i'(\De^u)$ be the right hand side of $\cM_i(\De^u)$ in this theorem. 
We show that $M_i$ in Theorem \ref{th:monod} appears as   
the representation matrix of $\cM_i'$ with respect to the 
column vector 
$$\mathbf{\Delta}^u=\tr (\De_{00}^u,\De_{10}^u,\dots, \De_{22}^u).$$
By Lemma \ref{lem:orth}, it is easy to see that 
$\cM_1'(\De^u_{0j})=\De^u_{0j}$ $(j=1,2,3)$. Note that 
\begin{align*}
\cM_1'(\De^u_{1j})&=\De^u_{1j}-(1-\b_1^{-1})(\de_{0j},\de_{1j},\de_{2j})
\begin{pmatrix}
\De^u_{10}\\
\De^u_{11}\\
\De^u_{12}\\
\end{pmatrix}=\b_1^{-1}\De^u_{1j},\\
\cM_1'(\De^u_{2j})&=\De^u_{2j}-(1-\b_2^{-1})(\de_{0j},\de_{1j},\de_{2j})
\begin{pmatrix}
\De^u_{20}\\
\De^u_{21}\\
\De^u_{22}\\
\end{pmatrix}=\b_2^{-1}\De^u_{2j},
\end{align*}
where $\de_{ij}$ is Kronecker's symbol. Thus we have the representation 
matrix $M_1$ from $\cM'_1$. Similarly we have 
the representation matrix $M_2$ from $\cM'_2$.
By the definition of $\cM_3'$, we have 
$$\cM'_3(\OO^u)=\frac{\b_1\b_2\b_3\b_4}{\a_1\a_2\a_3}\OO^u, \quad 
\cM'_3(\De^u)=\De^u $$ 
for any $\De^u\in H_4(T,\mathcal{L}_u)$ satisfying 
$\cI(\De^u,{\OO^u}^\vee)=0$.  
We have only to note that $\OO^u=(1,\dots,1)\mathbf{\Delta}^u$ and 
$$\cI(\De^u,{\OO^u}^\vee)=0 
\Longleftrightarrow
(w_{00},\dots,w_{22}) H\tr (1,\dots,1)=0$$
for $\De^u=(w_{00},\dots,w_{22})\mathbf{\Delta}^u$.
\end{proof}

\end{document}